\documentclass[12pt]{amsart}

\usepackage{graphicx}
\usepackage{amssymb}
\usepackage{amsthm}
\usepackage[mathscr]{eucal}

\input xy
\xyoption{all}

\theoremstyle{plain}
\newtheorem{theorem}{Theorem}[section]
\newtheorem{corollary}[theorem]{Corollary}
\newtheorem{lemma}[theorem]{Lemma}

\newtheorem{proposition}[theorem]{Proposition}
\newtheorem{fact}[theorem]{Fact}

\newtheorem{claim}[theorem]{Claim}

\theoremstyle{definition}
\newtheorem{definition}[theorem]{Definition}
\newtheorem{remark}[theorem]{Remark}

\newtheorem{notation}[theorem]{Notation}

\theoremstyle{remark}

\newcommand{\ob}{\textup{Ob}}
\newcommand{\mor}{\textup{Mor}}

\newcommand{\dom}{\operatorname{dom}}

\newcommand{\tp}{\operatorname{tp}}

\newcommand{\bdd}{\operatorname{bdd}}
\newcommand{\acl}{\operatorname{acl}}
\newcommand{\dcl}{\operatorname{dcl}}

\newcommand{\id}{\operatorname{id}}

\renewcommand{\phi}{\varphi}

\newbox\noforkbox \newdimen\forklinewidth
\forklinewidth=0.3pt \setbox0\hbox{$\textstyle\smile$}
\setbox1\hbox to \wd0{\hfil\vrule width \forklinewidth depth-2pt
 height 10pt \hfil}
\wd1=0 cm \setbox\noforkbox\hbox{\lower 2pt\box1\lower
2pt\box0\relax}
\def\unionstick{\mathop{\copy\noforkbox}\limits}

\def\nonfork_#1{\unionstick_{\textstyle #1}}

\setbox0\hbox{$\textstyle\smile$} \setbox1\hbox to \wd0{\hfil{\sl
/\/}\hfil} \setbox2\hbox to \wd0{\hfil\vrule height 10pt depth
-2pt width
               \forklinewidth\hfil}
\newbox\doesforkbox
\setbox\doesforkbox\hbox{\lower 2pt\box1 \lower
2pt\box2\lower2pt\box0\relax}
\def\nunionstick{\mathop{\copy\doesforkbox}\limits}

\def\fork_#1{\nunionstick_{\textstyle #1}}

\newcommand{\C}{\operatorname{\mathfrak{C}}}
\renewcommand{\P}{{\mathcal P}}
\newcommand{\Pm}{{\mathcal P}^{-}}

\newcommand{\B}{\mathcal{B}}

\newcommand{\Aut}{\operatorname{Aut}}

\newcommand{\G}{\mathcal{G}}

\def\a{\bar{a}}

\def\ov{\overline}
\def\Q{Q}

\def\A{\mathfrak{a}}
\def\B{\mathfrak{b}}
\def\cC{\mathcal{C}}

\def\CC{\cC}
\def\ind{\nonfork}
\def\CP{\P}
\def\be{\begin{enumerate}}
\def\ee{\end{enumerate}}

\def\bsigma{\mbox{\boldmath $\sigma$}}
\def\btau{\mbox{\boldmath $\tau$}}

\title[Amalgamation functors and boundary properties]{Amalgamation functors and boundary properties in simple theories}
\author{John Goodrick}
\author{Byunghan Kim}
\author{Alexei Kolesnikov}\thanks{The second author was supported by NRF grant 2010-0016044.
The third author was partially supported by NSF grant DMS-0901315.}

\begin{document}

\maketitle

\begin{abstract}
This paper continues the study of generalized amalgamation properties begun in \cite{DKY}, \cite{GK}, \cite{Hr}, and \cite{KKT}.  Part of the paper provides a finer analysis of the groupoids that arise from failure of $3$-uniqueness in a stable theory. We show that such groupoids must be abelian and link the binding group of the groupoids to a certain automorphism group of the monster model, showing that the group must be abelian as well.

We also study connections between $n$-existence and $n$-uniqueness properties for various ``dimensions'' $n$ in the wider context of simple theories. We introduce a family of weaker existence and uniqueness properties. Many of these properties did appear in the literature before; we give a category-theoretic formulation and study them systematically. Finally, we give examples of first-order simple unstable theories showing, in particular, that there is no straightforward  generalization of the groupoid construction in an unstable context.
\end{abstract}

\section*{Introduction}

Amalgamation properties have been one of the main tools in the study of simple theories from the beginning, starting in \cite{KP} with the Kim and Pillay's proof of the Independence Theorem for simple theories (this is the property that we call 3-existence). Further work on higher amalgamation properties was carried out in \cite{nsimple} and \cite{KKT}, and in \cite{DKY} it was shown that Hrushovski's group configuration theorem can be generalized to simple theories under the additional assumption of what we are calling $4$-existence in this paper.

While 3-existence holds in all simple theories, not all stable theories have $n$-existence property for $n\ge 4$. The first examples of this were suggested by Hrushovski (appeared in~\cite{DKY}), and in \cite{PS} Pastori and Spiga constructed a family of stable theories $T_n$ (for $n \geq 4$) which have $k$-existence for $k < n$ but fail to have $n$-existence.
In \cite{Hr}, Hrushovski has linked $4$-existence in stable theories with %binding groups of certain internal covers and
the existence of definable groupoids that are \emph{non-eliminable}, i.e., are not definably equivalent, in a category-theoretic sense, to a definable group. An explicit construction of a non-eliminable groupoid that witnesses the failure of 4-existence (or, equivalently, 3-uniqueness) in a stable theory appears in~\cite{GK}. We give the definition of a definable groupoid in Section~2 of this paper and refer the reader to the papers \cite{GK,Hr} for the notions related to eliminability of groupoids.

Section~2 of the current paper focuses on groupoids that arise from the failure of 3-uniqueness property in a stable theory. Any such failure is witnessed by a set of elements $W=\{a_1,a_2,a_3,f_{12},f_{23},f_{13}\}$ and a formula $\theta$ such that $f_{ij}\in \acl(a_ia_j)\setminus \dcl(a_i a_j)$ and $\theta$ witnesses that $f_{ij}$ is definable from $f_{ik}$ and $f_{jk}$. The complete set of properties is listed in Definition~\ref{full_symm_witness}. This set of elements, called the full symmetric witness, allows to construct a definable groupoid (the construction was carried out in~\cite{GK}). In this paper, we show that such a groupoid must have abelian binding groups $\mor_{\G}(a,a)$ and that the  composition law of the groupoid $\G$ directly corresponds to the formula $\theta$ (Proposition~\ref{coherence1}). We show that if $f'_{ij}\equiv _{a_ia_j} f_{ij}$, then the groupoid constructed from the witness $\{a_1,a_2,a_3,f'_{12},f'_{23},f'_{13}\}$ is definably isomorphic to the groupoid obtained from $W$.

We also show that the binding group of the groupoid is isomorphic to $\Aut(f_{12}/a_1a_2)$. This allows to conclude that, in a stable theory for any independent realizations $a$, $b$, $c$ of a non-algebraic type, the group $\Aut(\widetilde{ab}/\acl(a)\acl(b))$ is abelian. The symbol $\widetilde{ab}$ denotes the set $\acl(ab)\cap \dcl(\acl(ac),\acl(bc))$.

The remainder of the paper (Sections 3, 4, and 5) is devoted to the analysis of connections between generalized amalgamation properties of higher dimensions and in the broader context of simple theories. To motivate the analysis, let us give an informal description of amalgamation properties of ``dimension''~4.

For each three-element subset $s$ of $4=\{0,1,2,3\}$, let $\A(s)$ be a subset of the monster model of a first-order theory such that the family $\{\A(s)\}$ satisfies the compatibility and independence conditions. The compatibility conditions say, in particular, that if $t$ is a common subset of $s_1$ and $s_2$, then there is an elementary bijection between the subsets $\A_{s_1}(t)$ and $\A_{s_2}(t)$ of $\A(s_1)$ and $\A(s_2)$, respectively. The independence conditions say that the sets $\{\A_{s}(\{i\})\mid i\in s\}$ are independent for every 3-element set $s$. A separate condition demands that each $\A_{s}(\{i\})$ is an algebraically (or boundedly) closed set.

It is convenient to visualize the sets $\A(s)$, for 3-element subsets $s$, as triangles. The sets $\A_s(\{i\})$ for $i\in s$ are the vertices in the triangle; the compatibility conditions say that the corresponding vertices and edges have the same type. For example, the ``vertices'' $\A_{123}(\{3\})$ and $\A_{023}(\{3\})$ should have the same type and $\left\{\A_{123}(\{1\}),\A_{123}(\{2\}),\A_{123}(\{3\})\right\}$ should be an independent set (over the base $\A_{123}(\emptyset)$).

\begin{center}
\begin{picture}(0,0)%
\includegraphics{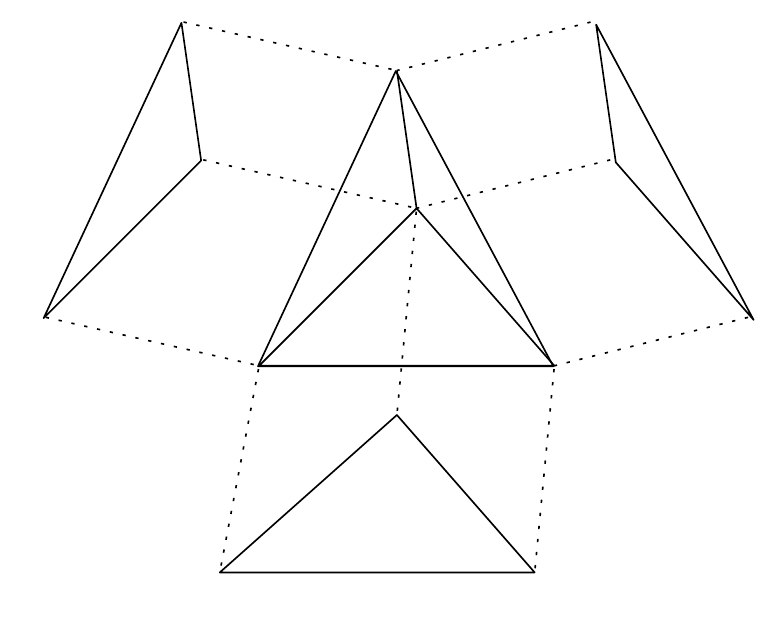}%
\end{picture}%
\setlength{\unitlength}{4144sp}%
\begin{picture}(3495,2916)(-374,-2164)
\put(1036,-1681){\makebox(0,0)[lb]{\smash{$\A(013)$}%
}}
\put(1486,-1186){\makebox(0,0)[lb]{\smash{$\A_{013}(3)$}%
}}
\put(1531,-421){\makebox(0,0)[lb]{\smash{3}%
}}
\put(1396,479){\makebox(0,0)[lb]{\smash{2}%
}}
\put(3106,-781){\makebox(0,0)[lb]{\smash{$\A_{123}(1)$}%
}}
\put(2431,569){\makebox(0,0)[lb]{\smash{$\A_{123}(2)$}%
}}
\put(1846, 29){\makebox(0,0)[lb]{\smash{$\A_{123}(3)$}%
}}
\put(136,-1996){\makebox(0,0)[lb]{\smash{$\A_{013}(0)$}%
}}
\put(2116,-1996){\makebox(0,0)[lb]{\smash{$\A_{013}(1)$}%
}}
\put(1036,-2086){\makebox(0,0)[lb]{\smash{$\A_{013}(01)$}%
}}
\put(2206,-1096){\makebox(0,0)[lb]{\smash{1}%
}}
\put(631,-1096){\makebox(0,0)[lb]{\smash{0}%
}}
\put(-359,-916){\makebox(0,0)[lb]{\smash{$\A_{023}(0)$}%
}}
\put(-134,569){\makebox(0,0)[lb]{\smash{$\A_{023}(2)$}%
}}
\put(541, 74){\makebox(0,0)[lb]{\smash{$\A_{023}(3)$}%
}}
\end{picture}%
\end{center}

%Each edge can be ``constructed'' from the vertices in two different ways: one can take the edge $\A_{012}(\{0,1\})$ to be the algebraic (or bounded) closure of $\A_{012}(\{0\})\cup \A_{012}(\{1\})$; or the edge can be simply the union of the two vertices. Similarly, each face (i.e., each of the triangles) can be the bounded closure of the union of the vertices, or the union of the bounded closures of all edges, or simply the union of the vertices. These differences turn out to be important, because they give rise to different amalgamation properties.

The property of 4-existence asserts that, if $\{\A(s)\mid s\subset 4,\ |s|=3\}$ is a compatible independent system of boundedly closed sets, the there is a set $\A(\{0,1,2,3\})$ and elementary embeddings of the sets $\A(s)$ into $\A(\{0,1,2,3\})$ that commute with the given elementary bijections of the system. The property of 4-uniqueness says that such a set $\A(\{0,1,2,3\})$ is unique up to isomorphism.  The precise definition of $n$-uniqueness is stated in the definitions in Section~1; we follow \cite{Hr} in formulating the amalgamation properties in category-theoretic language, and in particular we link $n$-uniqueness with natural isomorphism of functors.

The property of $(4,1)$-amalgamation ensures the existence of the set $\A(\{0,1,2,3\})$ when each ``face'' $\A(\{i,j,k\})$ is the union of its vertices; $(4,2)$-amalgamation says that the set $\A(\{0,1,2,3\})$ exists when each such face is the union of the closures of its ``edges'' $\A(\{i,j\}), \A(\{i,k\}),$ and $\A(\{j,k\})$.

The property $B(3)$ (where ``$B$'' stands for ``boundary'') is a weaker form of 3-uniqueness; it says that the union of any automorphisms of all the edges in a triangle that fix the vertices is an elementary map. The property is not new, it was used, for example, in~\cite{Hr, DKY} and later in \cite{GK}.

Here are some examples of known connections between different amalgamation properties; we take these as a starting point. Any stable theory has $(n,1)$-amalgamation for all $n\ge 2$. However, there are stable theories that fail $4$-existence (any such theory must fail $B(3)$ and $(4,2)$-existence). It was shown in~\cite{DKY} that, in a stable theory, the properties $B(k)$, $3\le k\le n$, imply $(n+1)$-existence. For simple theories, the random tetrahedron-free hypergraph (see \cite{KKT}) has $B(3)$ but fails $(4,1)$-amalgamation.

It was shown by Hrushovski in~\cite{Hr} that, for $n\ge 2$, the properties of $n$-existence and $n$-uniqueness imply $(n+1)$-existence for any simple theory. It is natural to ask whether there is a reverse implication. In Section 5, we construct an example of a simple theory that has $4$-existence (and any simple theory has 3-existence), but fails 3-uniqueness. In fact, the theory fails even the weaker property $B(3)$.

However, in Section~3 we obtain the following (Corollary~\ref{n-uniqueness}): if a simple theory has $(n-1)$-uniqueness and $(n+1)$-existence, then it has $n$-uniqueness. We prove this in two stages, first showing that $(n-1)$-uniqueness and $(n+1)$-existence imply the weaker property $B(n)$, and then showing that $(n-1)$-uniqueness with $B(n)$ imply $n$-uniqueness.
Thus, in any simple theory, the failure of $n$-uniqueness is linked to a failure of either $(n-1)$-uniqueness or of $(n+1)$-existence.

In Section~4, we show (Corollary~\ref{Bn_simplicity}) that if $T$ is simple and \emph{fails} to have $n$-existence for some $n \geq 4$, then there is a dichotomy: either $T$ fails to have one of the weak uniqueness properties $B(k)$ for some $k$ between $3$ and $n-1$, or else $T$ fails to have $(n,1)$-amalgamation.

Ideally, one would like to use the above dichotomy to show that \emph{any} failure of $4$-existence in a simple theory arises from the interpretability (in some suitable sense) of either a non-eliminable groupoid or a tetrahedron-free hypergraph.  However, an obstacle to proving this is that unlike in the stable case, the failure of $B(3)$ in a simple theory does not give rise to a non-eliminable definable groupoid, as we show in our first example in Section~5 of this paper.

The authors are grateful to the Maryland Logic Seminar for providing a friendly venue to present the preliminary results contained in this paper. We would like to thank Martin Bays for his interest and helpful conversations.

%The purpose of this paper is to extend the results of~\cite{GK} in the following three directions. The first section is devoted to the proof of the fact that, assuming $\le (n-1)$-uniqueness, the $(n+1)$-existence property is equivalent to $n$-uniqueness property. One direction, that $n$-uniqueness implies $(n+1)$-existence, was proved under a slightly different assumption by Hrushovski in~\cite{Hr}. In the second section, we sharpen some of the results about definable groupoids in a stable theory that fails $3$-uniqueness. In~\cite{GK} it was shown that groupoids arising from the failure of $3$-uniqueness cannot be centerless; we establish here that the groupoids must in fact be abelian. We also establish a much closer connection between the formula that binds the symmetric witness to failure of 3-uniqueness (a sort of ``groupoid configuration'' used to define the groupoid) and the composition in the resulting groupoid. Finally, we explore the generalized amalgamation properties in the simple case. As expected, the situation there is much less clear. We provide several examples that show that the simple case is quite different and finish by showing a connection between $n$-existence and $n$-existence for skeletons.

\section{The basic amalgamation properties: $n$-existence and $n$-uniqueness}

First, some notation. In this section of the paper, the default assumption is that the theory $T$ we work with is simple. When stronger results hold for stable theories, the stability assumption is stated explicitly. Let $\C$ be the monster model of $T$. For a simple unstable $T$ we work with $\C^{heq}$ and for a stable $T$ we work with $\C^{eq}$.

Bars over sets or tuples denote the algebraic closure of the set if the underlying theory is stable; and the bounded closure of the set if the theory is simple unstable. Although we have not essentially used it in the paper  the reader may wish to make a simplifying assumption that all theories considered here have elimination of hyperimaginaries; then the bars would denote the algebraic closure in both contexts.

A hat over an element means that it is omitted from a list.  For example, ``$\overline{a_0 \ldots \widehat{a_i} \ldots a_{n-1}}$'' means $``\acl(a_0, \ldots, a_{i-1}, a_{i+1}, \ldots, a_{n-1})$'' if we are dealing with a stable theory and $``\bdd(a_0, \ldots, a_{i-1}, a_{i+1}, \ldots, a_{n-1})$'' in the simple unstable context.

If we are working with a set $\{a_0, \ldots, a_{n-1}\}$ of elements that are independent over some base set $B$ and $u \subseteq \{0, \ldots, n-1\}$, we write ``$\overline{a}_u$'' for the set $$\overline{\{a_i : i \in u\} \cup B}.$$  For example, $\overline{a}_\emptyset$ is either $\acl(B)$ or $\bdd(B)$.

If $\C$ is a monster model of $T$ and $A$, $B$ are subsets of $\C$, then by $\Aut(A/B)$ we mean the group of elementary permutations of $A\cup B$ that fix $B$ pointwise.

Now we define the properties of $n$-existence and $n$-uniqueness.  Following \cite{Hr}, we define these in terms of solutions to independent amalgamation problems. Both the amalgamation problems and their solutions are functors between certain categories.

We let $n$ refer to the $n$-element set $\{0, \ldots, n-1\}$, and $\mathcal{P}^{-}(n) = \{s : s \subsetneq n\}$.  If $1 \leq i \leq n$, we use the notation ``$s \subseteq_i n$'' to mean that $s$ is an $i$-element subset of $n$.  If $S\subseteq \P(n)$ be closed under subsets, we view $S$ as a category, where the objects are the elements of $S$ and morphisms are the inclusion maps. Let $\cC$ be the category, where the objects are all algebraically closed subsets of $\C^{eq}$ (if we are working with a stable theory) or all boundedly closed subsets of $\C^{heq}$ (if the theory is simple unstable). In either case, the morphisms are the elementary maps, not necessarily surjective.

\begin{definition}
An \emph{$n$-amalgamation problem} is a functor $\A: \mathcal{P}^{-}(n) \rightarrow \cC$. A \emph{solution} to an $n$-amalgamation problem $\A$ is an extension to a functor $\A': \mathcal{P}(n) \rightarrow \cC$.
\end{definition}

\begin{definition}
If $S$ is a subset of $n$ closed under subsets and $\A:S\to \cC$ is a functor, then for $s \subseteq t \in S$, let the \emph{transition map} $$\A(s,t) : \A(s) \rightarrow \A(t)$$ be the image of the inclusion $s \subseteq t$.  When it is more convenient, we will also use the notation ``$\A_{s,t}$'' for the transition map $\A_{s,t}$.  So functoriality means that $\A_{t,u} \circ \A_{s,t} = \A_{s,u}$ whenever the composition is defined.
\end{definition}

\begin{notation}
If $S$ is a subset of $n$ closed under subsets, $\A:S\to \cC$ is a functor, and $s\subset t\in S$, we use the symbol $\A_t(s)$ to denote the subset $\A_{s,t}(\A(s))$ of $\A(t)$.

If $\A:\P(n)\to \cC$ is a functor, then $\A^-$ denotes the functor $\A\restriction \Pm(n)$.
\end{notation}

The following notion will be helpful in the proofs of Proposition~\ref{weak uq -> uq} and Theorem~\ref{implies B(n)}.

\begin{definition}\label{localization}
Suppose that $S\subset \P(n)$ be closed under subsets and $u$ is in $S$. Let $S|_{u}$ denote the following subset of $\P(n \setminus u)$:
$$
S|_{u}:=\{s \in \P(n \setminus u) \mid s \cup u \in S\}.
$$
If $S$ is as above and $\A:S\to \cC$ is a functor, define $\A|_{u}$ to be the functor on the set $S|_{u}$ such that $$\A|_{u}(s) = \A(s \cup u)$$ and the transition maps of $S|_{u}$ are the ones induced naturally from $\A$. We call the functor $\A|_{u}$ \emph{the localization of $\A$ at $u$}.
\end{definition}

\begin{definition}
Two solutions $\A'$ and $\A''$ of the $n$-amalgamation problem $\A$ are \emph{isomorphic} if there is an elementary map $\sigma: \A'(n) \rightarrow \A''(n)$ such that for any $s \subsetneq n$, $$\sigma \circ \A'_{s,n} = \A''_{s,n}.$$

\end{definition}

%\begin{definition}
%If $A$ is an $n$-amalgamation problem, then $A$ is \emph{over $B$} if $B = A(\emptyset)$ and for every $s \subsetneq n$, $A_{\emptyset,s}$ fixes $B$ pointwise.
%\end{definition}

%If we are looking at the solutions of $A$, clearly we may assume that $A$ is over $A(\emptyset)$ (just shift the $A(s)$'s by appropriate automorphisms).
%From now on we always assume an amalgamation problem $A$ is over $A(\emptyset)$ just to simplify the notation.

\begin{definition}
Suppose that $S$ is a subset of $n$ closed under subsets and $\A:S\to \cC$ is a functor.

\begin{enumerate}

%\item $\A(\emptyset)$ is boundedly closed;
\item We say that $\A$ is \emph{independent} if for every nonempty $s \in S$, $\left\{\A_{s} (\{i\}) : i \in s \right\}$ is an $\A_{s}(\emptyset)$-independent set.

\item We say that $\A$ is \emph{closed} if for every nonempty $s \in S$, $\A(s) = \ov{\cup \left\{ \A_{s} (\{i\}): i \in s \right\} }$.
\end{enumerate}

\end{definition}

\begin{remark}
If the transition maps are all inclusions in a functor $\A: S \rightarrow \cC$, then $\A$ is independent if and only if for every $s, t \in \dom(\A)$, $\A(t) \nonfork_{\A(t \cap u)} \A(u)$.
\end{remark}

\begin{definition}
1. $T$ has \emph{$n$-existence}, equivalently \emph{$n$-amalgamation} if every closed independent $n$-amalgamation problem has an independent solution.

2. $T$ has \emph{$n$-uniqueness} if every closed independent $n$-amalgamation problem $\A$ and any two closed independent solutions $\A'$ and $\A''$ are isomorphic. %That is, there is an isomorphism $\sigma_n:\A'(n)\to \A''(n)$ such that $\sigma_n \circ \A'_{s,n} = \A''_{s,n}$.

3. $T$ has \emph{$n$-complete amalgamation} if for every $k$ with $2 \leq k \leq n$, $T$ has $k$-existence.
\end{definition}

Note that the existence of non-forking extensions of types and the independence theorem implies that \emph{any} simple theory has both 2- and 3-existence. In addition, stationarity is equivalent to 2-uniqueness property, so any stable theory has $2$-uniqueness.

\begin{proposition}\label{n-uq_characterization}
For a simple theory $T$ and any $n\ge 2$, the following are equivalent:
\begin{enumerate}
\item
$T$ has $n$-uniqueness;
\item
If $\A$ and $\A'$ are closed independent functors from $\P(n)$ to $\cC$ and $\{\sigma_s\mid s\in \Pm(n)\}$ is a system of elementary bijections, $\sigma_s:\A(s)\to \A'(s)$, such that $\sigma_t\circ \A_{s,t} = \A'_{s,t}\circ \sigma_s$ for all $s\subset t \in \Pm(n)$, then the map $\bigcup_{s\subset_{n-1} n} \A'_{s,n}\circ \sigma_s\circ (\A_{s,n})^{-1}$ is an elementary map from a subset of $\A(n)$ to $\A'(n)$.
\end{enumerate}
\end{proposition}

\begin{proof}
Suppose $T$ has $n$-uniqueness, let $\A$, $\A'$ be independent functors, and let $\{\sigma_s\mid s\in \Pm(n)\}$ be a system of elementary bijections satisfying the properties in (2). Let $\A''$ be the functor defined as follows: $\A''$ coincides with $\A$ on $\Pm(n)$, $\A''(n):=\A'(n)$, and the transition maps $\A''_{s,n}$ are given by $\A''_{s,n}:=\A'_{s,n}\circ \sigma_s$. By $n$-uniqueness, there is an elementary bijection $\sigma_n:\A(n)\to \A''(n)$ that commutes with the transition maps. By construction, $\sigma_n$ extends $\bigcup_{s\subset_{n-1} n} \A'_{s,n}\circ \sigma_s\circ (\A_{s,n})^{-1}$, so the latter is an elementary map.

Conversely, suppose (2) holds and consider two functors $\A'$ and $\A''$ that give closed independent solutions to a closed independent amalgamation problem $\A$. Take $\sigma_s:=\id_{\A(s)}$ for all $s\in \Pm(n)$; then the condition (2) gives that $\bigcup_{s\subset_{n-1} n} \A''_{s,n}\circ (\A'_{s,n})^{-1}$ is a partial isomorphism from a subset of $\A'(n)$ to $\A''(n)$; it can be extended to the algebraic (or bounded) closures. This establishes $n$-uniqueness.
\end{proof}

Natural isomorphism of functors is a well-known category theory notion. We state a special case of this notion here.

\begin{definition}
Let $S\subset \P(n)$ be closed under subsets. Two functors $\A$ and $\B$ from $S$ to $\cC$ are \emph{naturally isomorphic} if for all $s\in S$ there are elementary bijections $\sigma_s\colon \A(s)\to \B(s)$ such that the following diagrams commute for all $t\subset s\in S$:
$$
\xymatrix{\A(t)\ar[d]^{\sigma_t} \ar[r]^{\A_{t,s}} & \A(s)\ar[d]^{\sigma_s} \\
\B(t)  \ar[r]_{\B_{t,s}} & \B(s)
}
$$
For each $s\in S$, the map $\sigma_s$ is called the \emph{component of the natural isomorphism $\sigma$ at $s$}.
\end{definition}

\begin{remark}
Proposition~\ref{n-uq_characterization} can be phrased as follows:
a simple theory $T$ has $n$-uniqueness if and only if for any two closed independent functors $\A$ and $\B$ from $\P(n)$ to $\cC$, any natural isomorphism between $\A^-$ and $\B^-$ can be extended to a natural isomorphism between $\A$ and $\B$.
\end{remark}

The following remarks explain why are we need to take arbitrary elementary embeddings as morphisms in the category $\cC$ (rather than allowing only inclusions). First, we state a definition.

\begin{definition}
Let $S$ be a subset of $\P(n)$ closed subsets. A functor $\A:S\to \cC$ is called \emph{untwisted} if for all $s\subset t \in S$ the embeddings $\A_{s,t}$ are inclusions.
\end{definition}

The following easy claims establish that in some cases we do not lose generality by considering untwisted functors, but in other cases the assumption that the functor is untwisted is almost as strong as the assumption of $n$-existence.

\begin{claim}
For each $n>1$, every (independent) functor $\A:\P(n)\to \cC$ is naturally isomorphic to an (independent) untwisted functor $\hat{\A}$.
\end{claim}

\begin{proof}
Indeed, we simply let $\hat{\A}(n):=\A(n)$ and $\hat{\A}(s):=\A_n(s)$. The natural isomorphism is given by $\eta_s=\A_{s,n}$.
If $\A$ is an independent functor, then so is $\hat{\A}$.
\end{proof}

\begin{claim}
Let $n>1$, and consider a closed independent functor $\A:\Pm(n)\to \cC$ (that is, $\A$ is a closed independent amalgamation problem). Then $\A$ is naturally isomorphic to an untwisted functor $\hat{\A}$ if and only if $\A$ has a solution.
\end{claim}

\begin{proof}
If $\A$ has a solution, we can recover an untwisted functor $\hat \A$ naturally isomorphic to $\A$ the same way we did in the preceding claim. For the converse, the algebraic (or bounded) closure of $\bigcup _{s\subsetneq n}\hat{\A}(s)$ provides a solution to the amalgamation problem.
\end{proof}

\section{Witnesses to the failure of $3$-uniqueness in stable theories}

%We begin the section by exploring some properties of definable groupoids that we will need later. We use the terminology of~\cite{GK}. We recall, in an abbreviated form, some of the definitions contained in the first section of that paper.  Throughout this section, we will assume that the theory $T$ is stable.

%Throughout this section, we work with finitary type-definable connected groupoids.

The purpose of this section is to study in detail how $3$-uniqueness can fail \emph{in a stable theory}.  We sharpen some of the results about definable groupoids in a stable theory that fails $3$-uniqueness. In~\cite{GK} it was shown that groupoids arising from the failure of $3$-uniqueness cannot be centerless; we establish here that, under additional assumptions on symmetric witness, the groupoids must in fact be abelian. We also establish a much closer connection between the formula that binds the symmetric witness to failure of 3-uniqueness (a sort of ``groupoid configuration'' used to define the groupoid) and the composition in the resulting groupoid.    Finally, we show that if $(a,b,c)$ is a Morley sequence in a stable theory and  $\widetilde{ab} = \dcl(\ov{ac},\ov{bc})\cap \acl(ab)$, then the group $\Aut(\widetilde{ab}/\ov{a}\,\ov{b})$ is abelian (Theorem~\ref{abelian_automorphism_group}).

%[Old introductory paragraph below]
%
%The purpose of this section is to extend the results of~\cite{GK}. We sharpen some of the results about definable groupoids in a stable theory that fails $3$-uniqueness. In~\cite{GK} it was shown that groupoids arising from the failure of $3$-uniqueness cannot be centerless; we establish here that, under additional assumptions on symmetric witness, the groupoids must in fact be abelian. We also establish a much closer connection between the formula that binds the symmetric witness to failure of 3-uniqueness (a sort of ``groupoid configuration'' used to define the groupoid) and the composition in the resulting groupoid.

Let us reiterate that throughout this section, we assume that $T$ is stable.

\subsection{Background}

We begin the section by exploring some properties of definable groupoids that we will need later. We use the terminology of~\cite{GK}. We recall, in an abbreviated form, some of the definitions contained in the first section of that paper.  Throughout this section, we will assume that the theory $T$ is stable.

\begin{definition}
A \emph{groupoid} is a non-empty category in which every morphism is invertible.

A groupoid is \emph{connected} if there is a morphism between any two of its objects.

A connected groupoid $\G$ is called \emph{finitary} if some (equivalently, every) group $G_a:=\mor_{\G}(a,a)$ is finite.

A groupoid $\G$ is \emph{(type-) definable} if the sets $\ob(\G)$ and $\mor(\G)$ are (type-) definable, as well as the  composition operation ``$\circ$'' and the domain, range and identity maps (respectively denoted by $i_0$, $i_1$, and $\textup{id}$).
\end{definition}

Throughout Section 2, all the groupoids are  finitary connected type-definable.

The notion of a symmetric witness to failure of 3-uniqueness was introduced in~\cite{GK}, but we repeat it here.  For the next definition (and the remainder of the section), recall our convention that if $\{a_1, \ldots, a_n\}$ is an $A$-independent set and $s \subseteq n = \{1, \ldots, n\}$, then $$\overline{a}_s = \bdd(A \cup \{a_i : i \in s\}).$$  For ease of notation, we write ``$\overline{a}_{ij}$'' instead of $\overline{a}_{\{i,j\}}$.

\begin{definition}\label{symm_witness}
A \emph{symmetric witness to non-$3$-uniqueness} (in a stable theory) is an independent   sequence $\{a_1, a_2, a_3\}$ of finite tuples over an algebraically closed  set $A$ together with elements $f_{12}$, $f_{23}$, and $f_{13}$ such that:
\begin{enumerate}
\item
$f_{ij} \in \overline{a}_{ij}$;
\item
$f_{12} \notin \dcl(\overline{a}_1 \overline{a}_2)$; and
\item
$a_1a_2f_{12} \equiv_A a_2a_3f_{23} \equiv_A a_1a_3f_{13}$;
\item
there is a formula $\theta(x,y,z)$ over $A$ such that $f_{12}$ is the unique realization of
$\theta(x, f_{23}, f_{13})$, the element $f_{23}$ is the unique realization of $\theta(f_{12}, y, f_{13})$, and
$f_{13}$ is the unique realization of $\theta(f_{12}, f_{23}, z)$.
\end{enumerate}
\end{definition}

The name is explained by the following theorem.

\begin{theorem}[Theorem~2.4 in~\cite{GK}]
If $T$ does not have $3$-uniqueness, then there is a set $A$ and a symmetric witness to non-$3$-uniqueness over $A$.
\end{theorem}

In~\cite{GK}, a symmetric witness to failure of 3-uniqueness is used to construct a definable non-eliminable groupoid. It was shown that an automorphism group of a groupoid constructed from a symmetric witness has to have a non-trivial center. In Proposition~\ref{abelian} we show that, for a suitable symmetric witness, the groupoid must in fact be abelian.
To do this, we need to strengthen the properties of a symmetric witness.

For the rest of Section 2, we suppress the related parameter sets to $\emptyset$.

\begin{definition}\label{full_symm_witness}
Let $W=\{a_1,a_2,a_3,f_{12},f_{23},f_{13}\}$ be a symmetric witness. We say that $W$ is a \emph{full symmetric witness} if in addition
\begin{enumerate}
\item[(5)]
$\tp(f_{12}/\ov{a}_1\ov{a}_2)$ is isolated by the type $\tp(f_{12}/a_1a_2)$.% and
%\item[(6)]
%$a_1,a_2\in \dcl(f_{12})$.
\end{enumerate}
Abusing notation slightly, we say that $(a,f)$ is a (full) symmetric witness if it can be expanded to a (full) symmetric witness $(a, b, c, f, g, h)$.
\end{definition}

It is not difficult to modify the construction of a symmetric witness to obtain a full symmetric witness to failure of 3-uniqueness. In fact, we note that the proof of Theorem~2.4 in~\cite{GK} gives a bit more: any tuple in the set $(\dcl(\ov{a_1a_3}\, \ov{a_2a_3})\cap \ov{a_1a_2})\setminus \dcl(\ov{a}_1\ov{a}_2)$ is a subtuple of some (possibly, more than one) full symmetric witness.

\begin{proposition}\label{full symm}
If $T$ does not have $3$-uniqueness, then there is a set $A$ and a full symmetric witness to non-$3$-uniqueness over $A$.

In fact, if $(a_1, a_2, a_3)$ is the beginning of a Morley sequence and $f$ is any element of $\ov{a_{12}} \cap \dcl(\ov{a_{13}}, \ov{a_{23}})$ which is not in $\dcl(\ov{a_1}, \ov{a_2})$, then there is some full symmetric witness $(a_1, f')$ such that $f \in \dcl(f')$.
\end{proposition}

\begin{proof}
We describe the modification of the proof of Theorem~2.4 in~\cite{GK} that gives the full symmetric witness. Throughout this argument, we refer to the notation from that proof.

First of all, we may assume that the element $c_{12}$ from that proof contains the element $f$ in the hypothesis in its definable closure.  Adding finitely many elements from the algebraic closures of $a_i$, $i=1,2,3$, to the tuples $a_i$ if necessary, we may assume that the formulas $\chi_c(x;a_1,a_2)$, $\chi_d(y; a_2,a_3)$, and $\chi_e(z;a_1,a_3)$ isolate the types of $c_{12}$, $d_{23}$, and $e_{13}$ over the algebraic closures of their parameters. That is, $\chi_c(x;a_1,a_2)$ isolates the type $\tp(c_{12}/\ov{a}_1\ov{a}_2)$, and so on.

The rest of the argument remains the same, we construct the tuple $f_{12}$ and add, if necessary, more elements from the algebraic closures of $a_i$ to make sure that the type $\tp(f_{12}/\ov{a}_1\ov{a}_2)$ is isolated by the type $\tp(f_{12}/a_1a_2)$. Then the elements $f_{12}$, $f_{23}$ and $f_{13}$ will form a full symmetric witness.
\end{proof}

\subsection{Abelian automorphism groups}

\begin{proposition}\label{abelian}
Suppose that $\G$ is a definable connected finitary groupoid with at least two distinct objects. Suppose that for some $a\ne b\in \ob(\G)$, for any $f,g\in \mor(a,b)$ we have $\tp(f/\ov{a}\,\ov{b})=\tp(g/\ov{a}\,\ov{b})$. Then the group $G_a$ is abelian for some (equivalently, for all) $a\in \ob(\G)$.
\end{proposition}

\begin{proof}
Since the groupoid is connected, all the groups of automorphisms of its objects are isomorphic. Let $a\ne b\in \ob(\G)$ be such that $\tp(f/\ov{a}\,\ov{b})=\tp(g/\ov{a}\,\ov{b})$ for all $f,g\in \mor(a,b)$. Suppose for contradiction that $G_a$ is not abelian. Let $\sigma,\tau\in G_a$ be such that $\sigma\circ \tau\circ \sigma^{-1} \ne \tau$. Choose an arbitrary $f\in \mor(a,b)$ and let $g=f\circ \sigma$. Since the groupoid is finitary, the automorphism group of each object is contained in the algebraic closure of the object; so to reach a contradiction it is enough to show that $\tp(f/aG_abG_b)\ne \tp(g/aG_abG_b)$. The types are indeed distinct since $f\circ\tau\circ f^{-1}$ and $g\circ \tau \circ g^{-1} = f\circ \sigma\circ \tau\circ \sigma^{-1}\circ f^{-1}$ are different elements of $G_b$.
\end{proof}

\begin{corollary}
The groupoid constructed from a full symmetric witness is necessarily abelian.
\end{corollary}

\begin{proof}
For independent $a,b\in\ob(\G)$, we know that there is a definable bijection between the set $\mor(a,b)$ and the set $\{f\mid (a,b,f)\equiv (a_1,a_2,f_{12})\}$. Thus, all the morphisms have the same type over $\ov{a}\, \ov{b}$. The statement now follows from Proposition~\ref{abelian}.
\end{proof}

Combining the above corollary with Proposition~\ref{full symm}, we get the following.

\begin{corollary}
If $T$ is a stable theory that fails 3-uniqueness, then there is an abelian definable groupoid in $T$.
\end{corollary}

\begin{definition}\label{binding-group}
Suppose that $\G$ is a finitary  abelian groupoid.  Let $\sigma\in G_a$ and $\sigma'\in G_b$ be arbitrary elements. We write $\sigma\sim \sigma'$ if for some (any) $f\in \mor(a,b)$ we have $f\circ \sigma \circ f^{-1}=\sigma'$. Then $\sim$ is a finite equivalence relation on the elements of the automorphism groups of the objects.

We will use the symbol $\bsigma$ to denote the equivalence class of the relation $\sim$. The set of all equivalence classes together with the operation inherited from the groups $G_a$ forms a group isomorphic to each of the groups $G_a$. We use the symbol $G$ to denote the group and call $G$ the \emph{binding group of the groupoid $\G$}.

The group $G$ naturally acts on the set $\mor(\G)$. For $f\in \mor_{\G}(a,b)$, the left action $\bsigma.f $ is given by the composition $\sigma\circ f$, where $\sigma$ is the unique element in $\bsigma\cap G_b$; the right action $f.\bsigma$ is given by $f\circ \sigma$ where $\sigma$ is the unique element in $\bsigma\cap G_a$.
\end{definition}

We collect some properties of the action of $G$ in $\mor(\G)$ in the following claim; all the properties follow immediately from the definitions.

\begin{claim}\label{action}
Let $\G$ be a finitary abelian groupoid and let $G$ be the set of equivalence classes with respect to the relation $\sim$ with the group operation inherited from the groups $G_a$, $a\in \ob(\G)$. Then for all $f\in \mor(a,b)$, $g\in \mor(b,c)$, for all $\bsigma,\btau\in G$ we have
\begin{enumerate}
\item
$\bsigma.f=f.\bsigma$;
\item
$(g\circ f).\bsigma = g\circ (f.\bsigma)$;
\item
$f.(\bsigma\circ \btau)=(f.\bsigma).\btau$.
\end{enumerate}
\end{claim}

\subsection{Coherence and multiple composition rules}

We now describe a somewhat surprising connection between the  formula $\theta$ that ``binds'' the symmetric witness and the composition in the definable groupoid constructed from the symmetric witness.
In~\cite{GK}, the construction of a definable non-eliminable groupoid from a symmetric witness to failure of 3-uniqueness uses two notions of composition. One is the auxiliary notion: the unique element $z$ such that $\theta(f_{12}, f_{23}, z)$ holds can be thought of as the ``composition'' of $f_{12}$ with $f_{23}$. The other is the composition in the groupoid defined on the equivalence classes of paths. It is natural to ask whether the two composition notions coincide, and whether a different choice of the elements realizing the same types as $f_{ij}$ would produce a different groupoid.

Let $W=\{a_1, a_2, a_3, f_{12}, f_{23}, f_{13}\}$ be a full symmetric witness to the failure of 3-uniqueness. In this section we show that, (1) for $W$, the two compositions coincide; (2) any choice of elements $f'_{ij}$ such that $\tp(a_ia_jf'_{ij})=\tp(a_1a_2f_{12})$ also gives a symmetric witness, $W'$, to failure of $3$-uniqueness; (3) the groupoid $\G'$ constructed from $W'$ is isomorphic to the groupoid $\G$ constructed from $W$; and (4) the objects and morphisms of $\G$ and $\G'$ are the same, but the composition rules may be different.

\begin{notation}\label{binding-group-action}
Let $W=\{a_1, a_2, a_3, f_{12}, f_{23}, f_{13}\}$ be a full symmetric witness to the failure of 3-uniqueness and let $\G$ be the definable groupoid constructed from $W$. By Lemma~2.14 of~\cite{GK}, there is an $a_ia_j$-definable bijection between the set $\Pi_{ij}:=\{f\mid f\equiv_{a_ia_j} f_{ij}\}$ and the set of morphisms $\mor(a_i,a_j)$. Given $f\models \tp(f_{ij}/a_ia_j)$, we use the symbol $[f]$ to denote the corresponding morphism from $a_i$ to $a_j$ in $\G$.

The bijection allows us to extend the action of $G$ on the set $\mor_{\G}$ to the set $\Pi_{ij}$ in the natural way: we let $f.\bsigma$ be the unique element $g\in \Pi_{ij}$ such that $[g]=[f].\bsigma$. The left action is defined in a similar way.
\end{notation}

\begin{proposition}\label{coherence1}
Suppose that $\G$ is a groupoid constructed from a full symmetric witness in a stable theory (we assume the base set $A=\emptyset$). Suppose that $a,b\in \ob(\G)$ are independent. Then $\models \theta(f_{ab},f_{bc},f_{ac})$ if and only if $[f_{bc}]\circ [f_{ab}]=[f_{ac}]$.
\end{proposition}

\begin{proof}
Let $f_{ab}$, $f_{bc}$, and $f_{ac}$ be a full symmetric witness; let $\G$ be the corresponding definable groupoid. Let $[g]$ be the unique element in $\mor(b,b)$ such that $[f_{bc}]\circ [g]\circ [f_{ab}]=[f_{ac}]$; let $f'_{ac}$ be the unique element in the $\acl(ac)$ such that $[f_{bc}]\circ [f_{ab}]=[f'_{ac}]$. We aim to show that $f'_{ac}=f_{ac}$ (this would imply that $[g]$ is the identity and that the first condition holds).

Since $[f_{bc}]\circ [f_{ab}]=[f'_{ac}]$, there are elements $f_{da}$, $f_{db}$, and $f_{dc}$ such that $\theta(f_{da},f_{ab},f_{db})$, $\theta(f_{db},f_{bc},f_{dc})$, and $\theta(f_{da},f'_{ac},f_{dc})$ hold. Therefore, we have
$$
[f_{ab}]\circ [g]\circ [f_{da}]=[f_{db}];\quad
[f_{bc}]\circ [g]\circ [f_{db}]=[f_{dc}];\quad
[f'_{ac}]\circ [g]\circ [f_{da}]=[f_{dc}].
$$
It follows that
$$
f_{dc}=[f'_{ac}]\circ [g]\circ [f_{da}]=[f_{bc}]\circ [g]\circ [f_{ab}]\circ [g]\circ [f_{da}],
$$
hence, using right cancelation in $\G$, we get $[f'_{ac}]=[f_{bc}]\circ [g]\circ [f_{ab}]$. So $f_{ac}=f'_{ac}$.
\end{proof}

The above proposition allows to obtain an interesting property of the formula $\theta$ that binds the symmetric witness.

\begin{corollary}
Suppose that $\G$ is a groupoid constructed from a full symmetric witness. Then for any four independent  $a,b,c,d\in \ob(\G)$, we have
$$
\theta(f_{da},f_{ab},f_{db})\land
 \theta(f_{db},f_{bc},f_{dc})\land
  \theta(f_{da},f_{ac},f_{dc})\to \theta(f_{ab},f_{bc},f_{ac}).
$$
\end{corollary}

\begin{proof}
This follows immediately from Proposition~\ref{coherence1} by associativity of the groupoid composition.
\end{proof}

\begin{corollary}\label{twisted witness}
Suppose that $W=\{a_1, a_2, a_3, f_{12}, f_{23}, f_{13}\}$ is a full symmetric witness to the failure of 3-uniqueness. Then for any $f'_{ij}\in \acl(a_ia_j)$ such that $f_{ij}\equiv_{a_ia_j} f'_{ij}$, the set
$W'=\{a_1, a_2, a_3, f'_{12}, f'_{23}, f'_{13}\}$ is also a full symmetric witness to the failure of 3-uniqueness.
\end{corollary}

\begin{proof}
let $W'=\{a_1, a_2, a_3, f'_{12}, f'_{23}, f'_{13}\}$ be a set of elements such that $f_{ij}\equiv _{a_ia_j} f'_{ij}$ for all $1\le i<j\le 3$. By stationarity, there is an automorphism of $\acl(a_1a_2a_3)$ that fixes $a_1a_2a_3$ and sends $f'_{23}f'_{13}$ to $f_{23}f_{13}$. Therefore, we may assume that $W'$ has the form $\{a_1, a_2, a_3, f'_{12}, f_{23}, f_{13}\}$. We need to show that $W$ is a full symmetric witness.

The properties (1)--(3) of a symmetric witness in Definition~\ref{symm_witness} are satisfied for $W'$. The property (5) of a full symmetric witness is immediate. To show that (4) holds as well, we need to construct an appropriate formula $\theta'$ that would ``bind'' $f'_{12}$, $f_{23}$, and $f_{13}$.

Let $\bsigma$ be the unique element of $G$ such that $f'_{12}=f_{12}.\bsigma$. It is routine to check that the formula
$$
\theta' (x,y,z):=\theta(x.\bsigma^{-1},y,z)
$$
is as needed.
\end{proof}

This raises the question: if $W=\{a_1, a_2, a_3, f_{12}, f_{23}, f_{13}\}$ and $W'=\{a_1, a_2, a_3, f'_{12}, f'_{23}, f'_{13}\}$ are two full symmetric witnesses to the failure of $3$-uniqueness, then are the groupoids constructed from $W$ and $W'$ isomorphic?  The answer turns out to be ``yes.''  This follows from results in \cite{Hr} about liaison groupoids of finite internal covers, but we note here that there is a more direct argument. Recall that if $(a_1, a_2, a_3, f_{12}, f_{23}, f_{13})$ is a full symmetric witness, then the symbol $\Pi_{12}$ denotes the set $\{f\mid f\equiv_{a_1a_2} f_{12}\}$.  For ease of notation below, we let ``$\Aut(f_{12} / a_1 a_2)$'' denote the group of all permutations of the set $\Pi_{12}$ induced by elementary bijections of $\overline{a_1 a_1}$ that fix $a_1 \cup a_2$ pointwise.

\begin{proposition}
\label{isomorphic_groupoids}
If $\G$ is the groupoid constructed from a full symmetric witness $(a_1, a_2, a_3, f_{12}, f_{23}, f_{13})$, then $\mor_{\G}(a_1, a_1) \cong \Aut(f_{12} / a_1 a_2)$.  Hence the isomorphism type of the groupoid $\G$ depends only on $\tp(a_1)$ and $\tp(f_{12}/a_1 a_2)$.
\end{proposition}

\begin{proof}
First note that since the symmetric witness is full, the groupoid $\G$ is abelian. Let $G$ be the binding group of $\G$ described in Definition~\ref{binding-group}. As we pointed out in Notation~\ref{binding-group-action}, there is a natural left group action of $G$ on the set $\Pi_{12}$. This action is regular.

Also, there is a natural left group action of $\Aut(f_{12}/ a_1 a_2)$ on $\Pi_{12}$: if $\sigma \in \Aut(f_{12} / a_1 a_2)$ and $g \in \Pi_{12}$, just let $\sigma . g = \sigma(g)$.  This action is transitive because any two elements of $\Pi_{12}$ are conjugate over $a_1 \cup a_2$, and the action is regular because for any two $g, h \in \Pi_{12}$, $h \in \dcl(g \cup a_1 \cup a_2)$ (using the fact that $a_1$ and $a_2$ are part of a full symmetric witness).

If $g \in \Pi_{12}$, $\sigma \in \Aut(f_{12} / a_1 a_2)$, and $h \in G$, then $$\sigma . (h . g) = \sigma  (h \circ g) = \sigma(h) \circ \sigma(g)$$ $$= h \circ \sigma(g) = h . (\sigma . g),$$ so the proposition is a consequence of the following elementary fact:

 \begin{claim}
 \label{commuting_group_actions}
 Suppose that $G$ and $H$ are two groups with regular left actions on the same set $X$ (recall: a group action is \emph{regular} if for every pair of points $(x,y)$ from $X$, there is a unique $g \in G$ such that $g . x = y$).  Furthermore, suppose that the actions of $G$ and $H$ commute: that is, for any $g$ in $G$, any $h$ in $H$, and any $x$ in $X$, $g . (h. x) = h. (g .x)$.  Then $G \cong H$.
 \end{claim}

% \begin{proof}
 %Fix any element $x$ of $X$.  Then by the regularity of the action of $G$, for any $h \in H$ there is a unique $\varphi(h) \in G$ such that $h^{-1} . x = \varphi(h) . x$.  By the regularity of the action of $H$, this function $\varphi: H \rightarrow G$ is a bijection.  Now for any $h_1, h_2 \in H$, $$(\varphi(h_1 h_2)). x = (h_1 h_2)^{-1}.x = h_2^{-1} . (h_1^{-1} . x) = h_2^{-1} . (\varphi(h_1) . x) = \varphi(h_1) . (h_2^{-1} . x)$$ $$= \varphi(h_1) . (\varphi(h_2) . x) = (\varphi(h_1) \varphi(h_2)).x.$$

 %By the regularity of the action of $G$, the equation $(\varphi(h_1 h_2)).x = (\varphi(h_1) \varphi(h_2)).x$ implies that $\varphi$ is a homomorphism.
 %\end{proof}

\end{proof}

Our final result of this section is that a certain automorphism group $\Gamma$ which is naturally associated with failures of $3$-uniqueness must be abelian.  First we set some notation.  Suppose that $p$ is a complete type over an algebraically closed set (and without loss of generality, $p \in S(\emptyset)$) and $a$ and $b$ are independent realizations of $p$.  Let $\widetilde{ab}$ denote the set $\dcl(\ov{ac},\ov{bc})\cap \acl(ab)$, where $c\models p$ is independent from $ab$. By stationarity, the set $\widetilde{ab}$ does not depend on the choice of $c$.  Our goal here is to describe the group $\Aut(\widetilde{ab} / \overline{a} \overline{b})$. We think of this group as the group of partial elementary maps from $\widetilde{ab}$ onto itself which fix $\overline{a} \cup \overline{b}$ pointwise.

First note that $\Aut(\widetilde{ab} / \overline{a} \overline{b})$ is nontrivial if and only if there is a witness $(a,b,f)$ to the failure of $3$-uniqueness, where $f \in \widetilde{ab}$.  Let $I:=\{f\in \widetilde{ab}\mid \textrm{$(a, b, f)$ is a symmetric witness}\}$.  We consider elements of $I$ only up to interdefinability.  For each $f \in I$, let $G_f = \Aut(f / \ov{a} \, \ov{b})$; as shown above, $G_f$ is isomorphic to the binding group of the groupoid constructed from $(a,f)$, and so $G_f$ is finite and abelian.  Also note that if $f$ and $g$ are interdefinable, then there is a isomorphism between $G_f$ and $G_g$.  We put a partial order $\leq$ on $I$ by declaring that $f \leq f'$ if and only if $f \in \dcl(f')$.

\begin{claim}
The set $\{G_f\mid f\in I\}$ is a directed system of groups.
\end{claim}

\begin{proof}
Suppose $f_1\in \dcl(f_2)$ (and $f_1,f_2\in I$). Then any automorphism of $\C$ that fixes $f_2$ also fixes $f_1$, so there is a natural projection $\pi_{f_2,f_1}:G_{f_2}\to G_{f_1}$.

%given by restriction of an automorphism in $\Aut(f_2/\ov{a}\ov{b})$ to an %automorphism of $f_1$ over the algebraic closures is well-defined (since %$\sigma\in G_{f_2}$ moves $f_1$ to another realization of %$\tp(f_1/\ov{a}\,\ov{b})$).

For a pair $f,g\in I$, then by Proposition~\ref{full symm}, there is a symmetric witness $h$ such that $(f,g) \in \dcl(h)$.
\end{proof}

Let $\Gamma$ be the inverse limit of the system $\{G_f\mid f\in I\}$.  (In the case where $I = \emptyset$, we let $\Gamma$ be the trivial group.)  Since each $G_f$ is a finite abelian group, the group $\Gamma$ is a profinite abelian group.

%{\bf [BK 10/07/16: Please amend the proof of 2.17 below to say that the proof works for an infinite tuples
%$a,b$, and it indeed shows that $\Gamma=\Aut(\widetilde{ab}/\ov{a}\,\ov{b}).]$}

\begin{theorem}
\label{abelian_automorphism_group}
With the notation above, the group $\Aut(\widetilde{ab}/\ov{a}\,\ov{b})$ is isomorphic to $\Gamma$, and in particular it is profinite and abelian.
\end{theorem}

\begin{proof}

We define a homomorphism $\varphi:\Aut(\widetilde{ab}/\ov{a}\,\ov{b})\to \Gamma$ by sending $\sigma \in \Aut(\widetilde{ab}/\ov{a}\,\ov{b})$ to the element represented by the function $f\in I \mapsto (\sigma\restriction f) \in G_f$. By Proposition~\ref{full symm}, this embedding is one-to-one.

Finally, suppose that $g \in \Gamma$ is represented by a sequence $\langle \sigma_f : f \in I \rangle$ such that for every $f \in I$, the map $\sigma_f$ is in $\Aut(f / \ov{a} \ov{b})$.  Then it follows from the compactness theorem of first-order logic that there is a $\sigma \in \Aut(\widetilde{ab} / \overline{a} \overline{b})$ such that $\varphi(\sigma) = g$.  Hence $\varphi$ is surjective.

%Therefore $\Aut(\widetilde{ab}/\ov{a}\,\ov{b})$ is isomorphic to a closed subgroup of $\Gamma$.  By a standard result in profinite groups (see, e.g.\, \cite{Ribes}), $\Aut(\widetilde{ab}/\ov{a}\,\ov{b})$ is profinite, and it is clearly abelian.
\end{proof}

\section{The boundary property $B(n)$ and amalgamation}

\begin{definition}
\begin{enumerate}

\item Let $T$ be a simple theory, let $A$ be a subset of the monster model of $T$ and let $n\ge 2$. We say that \emph{the property $B(n)$ holds over $A$} if for every $A$-independent set $\{a_0, \ldots, a_{n-1}\}$,
\begin{multline*}
\dcl(\overline{\widehat{a_0} \ldots a_{n-1} A}, \ldots, \overline{a_0 \ldots \widehat{a_{n-2}} a_{n-1} A}) \cap \overline{a_0 \ldots a_{n-2} A}
\\
= \dcl(\overline{\widehat{a_0} \ldots a_{n-2} A}, \ldots, \overline{a_0 \ldots \widehat{a_{n-2}} A}).
\end{multline*}

\item
We say that \emph{$B(n)$ holds for $T$} if $B(n)$ holds over every subset $A$ of the monster model.

\item If $n \geq 2$, $T$ is \emph{$B(n)$-simple} if it satisfies $B(k)$ for every $k$ with $2 \leq k \leq n$.

\end{enumerate}
\end{definition}

\begin{remark}
The $B$ in $B(n)$ stands for ``boundary.''  Note that $B(2)$ just says that if $a_0 \nonfork_A a_1$, then $\bdd(a_0 A) \cap \bdd(a_1 A) = \bdd(A)$, so \emph{every} simple theory is $B(2)$-simple.
\end{remark}

The next lemma generalizes the first part of Lemma~3.2 from \cite{Hr}.

\begin{lemma}
\label{B_n}
In any simple theory $T$, for any set $B$ and any $n \geq 3$, the following are equivalent:

\begin{enumerate}
\item $T$ has $B(n)$ over $B$.
\item For any $A$-independent set $\{a_0, \ldots, a_{n-1}\}$ and any $c \in \overline{a}_{\{0, \ldots, n-2\}}$,
$$
\tp(c / \overline{a}_{\{\widehat{0}, \ldots, n-2\}} \ldots \ov{a}_{\{0, \ldots, \widehat{n-2}\}}) \vdash \tp(c / \ov{a}_{\{\widehat{0}, \ldots, n-1\}} \ldots \ov{a}_{\{0, \ldots, \widehat{n-2}, n-1\}}).
$$

\item For any $A$-independent set $\{a_0, \ldots, a_{n-1}\}$ and any map $\sigma$ such that $$\sigma \in \Aut(\ov{a}_{\{0 \ldots n-2\}} / \overline{a}_{\{\widehat{0}, \ldots, n-2\}} \cup \ldots \cup \ov{a}_{\{0, \ldots, \widehat{n-2}\}}),$$ $\sigma$ can be extended to $\ov{\sigma} \in \Aut(\C)$ which fixes $$\ov{a}_{\{\widehat{0}, \ldots, n-1\}} \cup \ldots \cup \ov{a}_{\{0, \ldots, \widehat{n-2}, n-1\}}$$ pointwise.
\end{enumerate}
\end{lemma}

\begin{proof}

Without loss of generality, $A = \emptyset$.

(1) $\Rightarrow$ (2): Suppose that $c \in \overline{a}_{\{0, \ldots, n-2\}}$, and let $X$ be the solution set of $\tp(c / \ov{a}_{\{\widehat{0}, \ldots, n-1\}} \ldots \ov{a}_{\{0, \ldots, \widehat{n-2}, n-1\}})$.  Since $X$ is of bounded size, there is a code for $X$ in $\ov{a}_{\{0, \ldots, n-2\}}$, and also $X \in \dcl(\ov{a}_{\{\widehat{0}, \ldots, n-1\}}, \ldots, \ov{a}_{\{0, \ldots, \widehat{n-2}, n-1\}})$.  So by $B(n)$, $X \in \dcl(\overline{a}_{\{\widehat{0}, \ldots, n-2\}}, \ldots, \ov{a}_{\{0, \ldots, \widehat{n-2}\}})$, and (2) follows.

(2) $\Rightarrow$ (1): If $c \in \ov{a}_{\{0, \ldots, n-2\}} \cap \dcl(\ov{a}_{\{\widehat{0}, \ldots, n-1\}}, \ldots, \ov{a}_{\{0, \ldots, \widehat{n-2}, n-1\}})$, then by (2), $\tp(c / \overline{a}_{\{\widehat{0}, \ldots, n-2\}} \ldots \ov{a}_{\{0, \ldots, \widehat{n-2}\}})$ has a unique solution, that is, $c \in \dcl(\overline{a}_{\{\widehat{0}, \ldots, n-2\}}, \ldots, \ov{a}_{\{0, \ldots, \widehat{n-2}\}})$.

(2) $\Rightarrow$ (3): (2) implies that any such $\sigma$ can be extended to an elementary permutation of $\ov{a}_{\{0, \ldots, n-1\}}$ which fixes $\ov{a}_{\{\widehat{0}, \ldots, n-1\}} \cup \ldots \cup \ov{a}_{\{0, \ldots, \widehat{n-2}, n-1\}}$ pointwise, and from there use saturation of $\C$ to extend to $\ov{\sigma} \in \Aut(\C)$.

(3) $\Rightarrow$ (2) is immediate.

\end{proof}

\begin{definition}
If $A\subset B\subset \C$, by an \emph{automorphism of $B$ over $A$} we mean an elementary map from $B$ onto itself which fixes $A$ pointwise.
\end{definition}

\begin{definition}
\label{relative_uniqueness}
Let $B$ be a set.  We say that \emph{relative $n$-uniqueness holds over $B$} if for every $B$-independent collection of boundedly  closed sets $a_0,\dots,a_{n-1}$ and for any set $\{\sigma_u \mid u \subset_{n-1} n\}$ such that:
\begin{enumerate}
\item
$\sigma_u$ is an automorphism of $\overline{a}_u$; and
\item
For all $v\subsetneq u$, $\sigma_u$ is an identity on $\overline{a}_v$,
\end{enumerate}
we have that $\bigcup _{u} \sigma_{u}$ is an elementary map (i.e., can be extended to the automorphism of $\C$).

A theory $T$ has the \emph{relative $n$-uniqueness property} if it has the relative $n$-uniqueness property over any set $B$.
\end{definition}

\begin{remark}
\label{st.rel2}
Relative $2$-uniqueness says: if $a_0$ and $a_1$ are boundedly  closed sets containing $B$ such that $a_0 \nonfork_B a_1$, and for $i = 0, 1$, $\sigma_i \in \Aut(a_i / \bdd(B))$, then $\sigma_0 \cup \sigma_1$ is an elementary map.  So if $T$ is stable, then the stationarity of strong types implies that $T$ has relative $2$-uniqueness.

\end{remark}

In the stable case, relative $2$-uniqueness was used substantially in the groupoid construction in \cite{GK}, so if we want to generalize these results simple unstable case, we will need to get around the failure of relative $2$-uniqueness somehow.  In fact, it turns out that relative $2$-uniqueness is the same as stability, as we will show next.

First, some notation: we say that $p=\tp(a/A)$ is Lascar strong type if it is an amalgamation base, equivalently for any $b\models p$,
$a\equiv_A^Lb$ holds; or $a\equiv_{\bdd(A)}b$ holds; or $p\vdash \tp(a/\bdd(A))$.

The next fact is folklore, and is not hard to establish directly from the definitions, so we omit the proof.

\begin{fact}
\label{dcl}
 Assume $d$, $b$  are given (possibly infinite) tuples.  The following are equivalent:
\begin{enumerate}
\item
$\tp(d/b)$ is a Lascar strong type.

\item For any $f\equiv_bf'$ with $f,f'\in\bar b$, we have
$f\equiv_{db}f'$.
\item
For any $f\in \bdd(b)-\dcl(b)$, we have $f\notin\dcl(db)$.

\item $\bdd(b)-\dcl(b)\subseteq \bdd(db)-\dcl(db)$.

\item $ \dcl(db)\cap \bdd(b)\subseteq \dcl(b)$.

\item $ \dcl(db)\cap \bdd(b)= \dcl(b)$.

\end{enumerate}
\end{fact}

%\begin{proof}
%(1)$\Rightarrow$(2) Assume  $f\equiv_bf'$ for  $f,f'\in\bar b$.
%Then there is a $d'$ such that $d'f'b\equiv dfb$. Due to (1), $d'\equiv_{\bar b}d$.
%Thus $df'b \equiv d'f'b \equiv dfb$.

%(2)$\Rightarrow$(3) Clear.

%(3)$\Rightarrow$(1) If $\tp(d/b)$ is not a Lascar strong type, then there  are $d'\equiv_b d$
%and $f\in \bar b$ such that $ dbf\not\equiv d'bf$. Thus there is
%$f'\in \bar b$ such that $dbf'\equiv d'bf$, and so $ dbf\not\equiv dbf'$.
%Let $g$ be the set of solutions  of $tp(f/bd)$.
%Then $g\in \dcl(db)\cap \acl(b)$, but $g\notin\dcl(b)$, since for $g'$ with
%$gf\equiv_b g'f'$, as $f'\notin g$, we have $g\ne g'$.

%(3)$\Leftrightarrow$(4)$\Leftrightarrow$(5)$\Leftrightarrow$(6) Clear.
%\end{proof}

Due to Fact~\ref{dcl}, the property $B(3)$ (over $\emptyset$) is equivalent to the statement that
for independent $\{a,b,c\}$, $\tp(\ov{ac}, \ov{bc}/\bar a\bar b)$ is a Lascar strong type.

\begin{lemma}
The following are equivalent for a simple theory $T$:

\begin{enumerate}
\item[(a)]
Relative 2-uniqueness.
\item[(b)]
For any $B$-independent $\{a,b\}$, we
have $$\tp(\ov{bB}/\bdd(B)b)\vdash \tp(\ov{bB}/\bdd(Ba)b).$$
\item[(c)]
Any nonforking extension of a Lascar strong type is a Lascar strong type.
\item[(d)] For $a\ind_B b$ with $B=\bdd(B)$, we have
$\dcl(\bdd(Ba), b)\cap \bdd(Bb)= \dcl(Bb)$.
\item[(e)]
 $T$ is stable.
\end{enumerate}

\end{lemma}

\begin{proof}
%Recall that elimination of hyperimaginaries means that up to interdefinability bounded closures are the same as %algebraic closures (in $T^{heq}$) -- see \cite{wagner} for more details.
(a)$\Rightarrow$(b) Assume (a). Let $B$-independent $\{a,b\}$
be given where $B=\bdd(B)$.  For any enumeration $\ov{bB}'$ such that
$\ov{bB}'\equiv_{Bb}\ov{bB}$, due to (a) we have
$\ov{bB}'\equiv_{\bdd(Ba)b}\ov{bB}.$
Hence (b) holds.

(b)$\Rightarrow$(c) Let a Lascar strong type $p=\tp(a/B)$ be given where $B$ is boundedly closed.
Given $b$ with $a\ind_B b$, we want to show $q=tp(a/Bb)$ is a Lascar strong type. Namely for $a'\models q$,
we shall see $a\equiv_{\ov{Bb}}a'$. Now  let $\ov{Bb}'$ be an image
of  $\ov{Bb}$ under a $Bb$-automorphism sending $a$ to $a'$. Hence $a\ov{Bb}\equiv a'\ov{Bb}'$.
But due to (b),
$\ov{Bb}\equiv_{a'B}\ov{Bb}'$. Thus $a\ov{Bb}\equiv a'\ov{Bb}$ as desired.

(c)$\Rightarrow$(d) This follows from Fact~\ref{dcl} above.

(d)$\Rightarrow$(e) Suppose $T$ is simple and unstable.  By instability, there is a type $p = \tp(b/M)$ over a model $M$ with two distinct coheir extensions over $Ma$ (for some $a$).  In other words, there are realizations $b$ and $b'$ of $p$ and a formula $\varphi(x; m a)$ over $Ma$ such that $\models \varphi(b;ma) \wedge \neg \varphi(b'; ma)$.  By the independence theorem, we may assume that $a \ind_M b b'$.  Let $d=\{b,b'\}$.  Then $a \ind_M d$ and $b\in \dcl(dMa)\cap \acl(dM)$, but $b \notin \dcl(Md)$, so (d) fails over $M$.

(e)$\Rightarrow$(a) Remark~\ref{st.rel2} says it.

%(a)$\Leftrightarrow$(d) $\Leftarrow$: If (d) fails, pick $c \in \acl(B, b) \setminus \dcl(B, b)$ such that $c \in %\dcl(\acl(B,a), b)$, and pick $c' \equiv_{Bb} c$ such that $c' \neq c$.  Then if $\sigma_1 \in \mbox{Aut}(\overline{B %b} / Bb)$ maps $c$ to $c'$ and $\sigma_2$ is the identity map on $\acl(B, a)$, the map $\sigma_1 \cup \sigma_2$ is %not elementary, contradicting relative $2$-uniqueness.

\end{proof}

Our next goal is to describe the connections between $n$-uniqueness, $n$-existence, and $B(n)$ for various values of the dimension $n$.  First, recall the following  fact (Lemma~3.1(2) from \cite{Hr}):

\begin{fact}
\label{n_ex_n_uq}
(Hrushovski) If $T$ is simple and $T$ has $n$-existence and $n$-uniqueness for some $n \geq 2$, then $T$ has $(n+1)$-existence.
\end{fact}

The $n=3$ case of the next theorem was proved in \cite{Hr} as Lemma~3.2.

\begin{theorem}
\label{B_n_rel_uniqueness}
If $T$ is simple, $n \geq 3$, and $B$ is any set, then $T$ has relative $n$-uniqueness over $B$ if and only if $T$ has $B(n)$ over $B$.
\end{theorem}

\begin{proof}
As usual, we may assume that $B = \emptyset$.

$\Rightarrow$: Suppose that $B(n)$ fails, and the failure is witnessed by the independent set $\{a_0, \ldots, a_{n-1}\}$.  So there is an element $e \in \acl(a_0, \ldots, a_{n-2})$ such that $$e \in \dcl(\overline{\widehat{a_0} \ldots a_{n-1}}, \ldots, \overline{a_0 \ldots \widehat{a_{n-2}} a_{n-1}}) \setminus \dcl(\overline{\widehat{a_0} \ldots a_{n-2}}, \ldots, \overline{a_0 \ldots \widehat{a_{n-2}}}).$$  So there is a map $\sigma_{n-1} \in \Aut(\overline{a_0 \ldots a_{n-2}})$ such that $\sigma_{n-1}$ fixes each of the sets $\overline{a_0 \ldots \widehat{a_i} \ldots a_{n-2}}$ pointwise but $\sigma_{n-1}(e) \neq e$.  For each $i$ between $0$ and $n-2$, let $\sigma_i$ be the identity map on $\acl(a_0 \ldots \widehat{a_i} \ldots a_{n-1})$.  Then $\sigma_0 \cup \ldots \cup \sigma_{n-1}$ is not elementary.

$\Leftarrow$: First we fix some notation.  For $i$ between $0$ and $n-1$, let $$A_i = \ov{a_0 \ldots \widehat{a_i} \ldots a_{n-1}},$$ $$B_i = \bigcup_{j \neq i} \ov{a_0 \ldots \widehat{a_i} \ldots \widehat{a_j} \ldots a_{n-1}},$$ $$C_i = \bigcup_{j \neq i} \ov{a_0 \ldots \widehat{a_j} \ldots a_{n-1}}.$$  Suppose that for each $i$, $\sigma_i \in \Aut(A_i / B_i)$.  Then by $B(n)$ and Lemma~\ref{B_n}, there are maps $\ov{\sigma_i} \in \Aut(\C / C_i)$ extending $\sigma_i$.  Let $\sigma := \ov{\sigma_{n-1}} \circ \ldots \circ \ov{\sigma_0}$.  Then if $j \neq i$, $\ov{\sigma_j}$ fixes $A_i$ pointwise, so $\sigma \upharpoonright A_i = \sigma_i$.  Therefore $\sigma_0 \cup \ldots \cup \sigma_{n-1}$ is elementary, proving relative $n$-uniqueness.

\end{proof}

Next, we give a category-theoretic reformulation of the property $B(n)$.

\begin{definition}
Let $n\ge 2$ and let $\A:\P(n)\to \cC$ be a functor which is closed, independent, and untwisted.  A \emph{one-side twisting of $\A$} is a natural isomorphism $\eta: \A^- \to \A^-$ such that all the components of $\eta$, except possibly one, are identity maps.

\end{definition}

\begin{remark}
By Lemma~\ref{B_n} above, the property $B(n)$ for  a simple theory is equivalent to the condition that every one-side twisting of an untwisted closed independent functor $\A: \P(n) \to \cC$ can be extended to a natural isomorphism from $\A$ to itself.  In proofs below, we will use this as the definition of $B(n)$.
\end{remark}

The following claim describes what a one-side twisting looks like.

\begin{claim}
Let $\A:\P(n)\to \cC$ be a closed untwisted independent functor and let $\eta$ be a one-side twisting of $\A$.  Then
\begin{enumerate}
\item
for any $s\in \Pm(n)$, if $\eta_s=\id(\A(s))$, then for all $t\subset s$ we have $\eta_t=\id(\A(t))$;
\item
if $\eta_s\ne \id(\A(s))$, then $s$ is an $(n-1)$-element subset of $n$.

\end{enumerate}
\end{claim}

\begin{proposition}\label{weak uq -> uq}
Let $n\ge 3$ and let $T$ be a simple theory that has $(n-1)$-uniqueness. Then $n$-uniqueness holds if and only if $B(n)$ holds.
\end{proposition}

\begin{proof}
By Lemma~\ref{B_n}, it is clear that $n$-uniqueness implies $B(n)$ for any theory without the assumption of $(n-1)$-uniqueness.

We prove the converse. Let $\A:\P(n)\to \cC$ be a closed independent functor. Without loss of generality, we may assume that $\A$ is untwisted. Let $\A'$ be an arbitrary functor that extends $\A^-$. We need to construct an elementary bijection $\sigma:\A'(n)\to \A(n)$ such that $\sigma \circ \A'_{s, n}=\A_{s,n}$ for all $s\in \Pm(n)$.

By $(n-1)$-uniqueness, the natural isomorphism $\nu:(\A'|_{\{n-1\}})^-\to (\A|_{\{n-1\}})^-$ given by the identity components can be extended to a natural isomorphism $\bar \nu:\A'|_{\{n-1\}}\to \A|_{\{n-1\}}$.

\begin{claim}
For $s\in \Pm(n)|_{\{n-1\}}$, let $\eta_s:=\id(\A(s))$, and let $\eta_{n-1}:=\bar \nu_{n-1}\circ \A'_{n-1,n}$. Then the  collection $\eta=\{\eta_s\mid s\in \Pm(n)\}$ is a one side twisting of $\A$.
\end{claim}

\begin{proof}
We only need to check that $\eta$ is a natural isomorphism from $\A$ to itself, i.e., that the component $\eta_{n-1}$ commutes with the identity components of $\eta$. Let $s$ be a proper subset of $n-1$. Then
$\A_{s,n-1}\circ \eta_s=\id(\A(s))$. On the other hand,
\begin{multline*}
\eta_{n-1}\circ \A_{s,n-1} = \eta_{n-1}\circ \id(\A(s)) =
\bar \nu_{n-1}\circ \A'_{n-1,n}\circ \id(\A(s))
\\
=\bar\nu_{n-1}\circ \A'_{s,n}=\A_{s,n}\circ \nu_s=\id(\A(s)).
\end{multline*}
Thus, $\eta_{n-1}\circ \A_{s,n-1} = \A_{s,n-1}\circ \eta_s$.
\end{proof}

By $B(n)$, there is an elementary map $\bar \eta$ that extends the natural isomorphism $\eta$. It is now easy to check that the map $\sigma:=\left(\bar\eta_n\right)^{-1} \circ \bar\nu_{n-1}$ is as needed.
\end{proof}

That $T$ has $(n-1)$-uniqueness is essential in Proposition~\ref{weak uq -> uq}, as for example the random ternary graph has $B(3)$, but it has neither 3-uniqueness nor 2-uniqueness.
Note that 2-uniqueness is equivalent to stability.

\begin{theorem}\label{implies B(n)}
Suppose that $T$ is a simple theory  and $T$ has the $(n-1)$-uniqueness property and the $(n+1)$-existence property for some $n\ge 3$. Then $T$ has $B(n)$.
\end{theorem}

\begin{proof}
The plan is to assume that $B(n)$ fails but $(n+1)$-existence holds, and show that $(n-1)$-uniqueness must fail. So let $\A:\P(n)\to\cC$ be a closed untwisted independent functor and let $\eta$ be a one side twisting of $\A$ that cannot be extended to a natural isomorphism from $\A$ to itself. Without loss of generality, we may assume that the component $\eta_{n-1}$ that maps $\A(n-1)$ to itself is not an identity map.
This means that we can pick an element $f \in \A(n-1)$ such that $f\in\dcl(g_0...g_{n-2})$ where $g_i =\A(n \setminus \{i\})$,  and $f'_0\not\equiv_{\overline{g}}f$ where $f'_0=\eta^{-1}_{n-1} (f)\ne f$.
%For simplicity, we may further assume that $f \in \dcl(g_0, \ldots, g_{n-2})$.  (If it is not, then we can replace %$f$ by a name $\tilde{f}$ for the set of all realizations of $\theta(x; \overline{g})$.)

Define the independent functor $\A':\P(n)\to \cC$ by letting $\A'(s):=\A(s)$ for all $s\in \P(n)$ and letting all the transition maps be identity maps, except for one: we let $\A'_{n-1,n}:=\eta_{n-1}$.

Next we pick a closed untwisted independent functor $\widehat{\A} : \P(n+1) \rightarrow \cC$ which extends $\A$ and such that $\widehat{\A} \upharpoonright \P( (n+1) \setminus \{n-1\})$ is isomorphic to $\A$ (via a natural isomorphism that associates $s \cup \{n\}$ to $s \cup \{n-1\}$ for any $s \subseteq (n-1)$).  So, if we write
$g_i^* = \widehat{\A}((n+1) \setminus \{i,n-1\})$, then
\begin{equation}
\label{star}\vec g := g_0...g_{n-2}\equiv_{\widehat{\A}(n-1)}\vec g^*=g^*_0...g^*_{n-2}.
\end{equation}
%$\widehat{\A}(\{n\}) \equiv \A(\{n-1\})$.
Since $\widehat{\A}(n-1)={\A}(n-1)$, in particular $\vec g\equiv_{f}\vec g^*$.

By $(n+1)$-existence, we can construct another closed independent functor $\widehat{\A'} : \P(n+1) \rightarrow \cC$ such that

\begin{enumerate}
\item If $s \subsetneq n+1$, then $\widehat{\A'}(s) = \widehat{\A}(s)$;
\item If $s \subseteq t \subsetneq n+1$ and $(s,t) \neq (n-1, n)$, then $\widehat{\A'}_{s,t} = \id_s$; and
\item $\widehat{\A'}_{n-1, n} = \eta$.
\end{enumerate}

Note that this means that the ``face'' $\widehat{\A'} \upharpoonright \P(n)$ is equal to $\A'$.

Finally, let $\B = \widehat{\A}|_{\{n-1,n\}}$ and $\B' = \widehat{\A'}|_{\{n-1,n\}}$, and we claim that these witness the failure of $(n-1)$-uniqueness.  By definition, $\B^{-} = (\B')^-$, so we will assume towards a contradiction that there is an elementary map $\sigma: \widehat{\A}(n+1) \rightarrow \widehat{\A'}(n+1)$ such that for any $s \subsetneq n - 1$, $$\sigma \circ \widehat{\A}_{s \cup \{n-1, n\}, n+1} = \widehat{\A'}_{s \cup \{n-1, n\}, n+1}.$$  Since $\widehat{\A}$ is untwisted, this last equation can be rewritten as
\begin{equation}
\label{sigma}\sigma \upharpoonright \widehat{\A}(s \cup \{n-1, n\}) = \widehat{\A'}_{s \cup \{n-1, n\}, n+1}. \end{equation}

%Pick elements $g_0^*, \ldots, g_{n-2}^*$ such that $g_i^* \in \A((n+1) \setminus \{i,n-1\})$ and %\begin{equation}\label{star} \tp(g_0^* \ldots g_{n-2}^* \widehat{\A}(\{n\}) / \widehat{\A}(n-1)) = \tp( g_0 \ldots %g_{n-2} \widehat{\A}(\{n-1\}) / \widehat{\A}(n-1)). \end{equation}

Now, for $i$ between $0$ and $n-2$, write
 $$g'_i := \widehat{\A'}_{n \setminus \{i\}, n+1} (g_i)=\widehat{\A'}_{ n+1}(n \setminus \{i\}),$$ and
 similarly $$g''_i := \widehat{\A'}_{(n+1) \setminus \{i,n-1\}, n+1} (g_i^*)
 =\widehat{\A'}_{ n+1}((n+1) \setminus \{i,n-1\}).$$
 By equation~\ref{sigma} above plus the fact that $\widehat{\A'}_{n \setminus \{i\}, (n+1) \setminus \{i\}}$ is an identity map, it follows that $g'_i = \sigma(g_i)$, and similarly $g''_i = \sigma(g^*_i)$. Thus $\sigma$ witnesses
  $\vec g\vec g^*\equiv \vec g'\vec g''$. Note also by equation~\ref{star} above, $f\vec g\equiv f\vec g^*$. Hence there must exist the unique element $f'=\sigma(f)$  such that
  $$f\vec g\equiv f'\vec g'\equiv f'\vec g''.$$

On the other hand, due to   compatibility,
$$\widehat{\A'}_{n, n+1}\restriction \vec g =\sigma \restriction \vec g.$$
Then since $f\in\dcl(\vec g)$, we have $\widehat{\A'}_{n, n+1}(f\vec g)=\sigma(f\vec g)=f'\vec g'$,
and so $f'=\widehat{\A'}_{n, n+1}(f)=(\widehat{\A'}_{n, n+1} \circ \widehat{\A'}_{n-1,n})(f'_0)
=\widehat{\A'}_{n-1, n+1} (f'_0).$ Also we have
$$\widehat{\A'}_{(n+1) \setminus \{n-1\}, n+1}\restriction \vec g^* =\sigma \restriction \vec g^*.$$ Thus similarly $f'$ must be equal to $\widehat{\A'}_{(n+1) \setminus \{n-1\}, n+1}(f)$.
But since this time $ \widehat{\A'}_{n-1,(n+1) \setminus \{n-1\}}$
is an identity map, we have
$$\widehat{\A'}_{(n+1) \setminus \{n-1\}, n+1}(f)=(\widehat{\A'}_{(n+1)\setminus \{n-1\}, n+1} \circ \widehat{\A'}_{n-1,(n+1) \setminus \{n-1\}})(f),$$
and so
$f'=\widehat{\A'}_{n-1, n+1} (f)$. This leads a contradiction as $f\ne f'_0$ and
$\widehat{\A'}_{n-1,n+1}$ is injective.

\end{proof}

\begin{corollary}
\label{n-uniqueness}
Suppose $T$ is simple, and has $(n-1)$-uniqueness and $(n+1)$-existence properties for some $n\ge 3$. Then $T$ has the $n$-uniqueness property.
\end{corollary}

\begin{proof}

By the previous theorem and Proposition~\ref{weak uq -> uq}.

\end{proof}

%\begin{corollary}
%Suppose $T$ is stable and $T$ has $k$-uniqueness for all $2 \leq k < n$ (where $n \geq 2$).  Then $T$ has %$n$-uniqueness if and only if $T$ has $(n+1)$-existence.
%\end{corollary}

%\begin{proof}

%The right-to-left direction is by Corollary~\ref{n-uniqueness}.  For the left-to-right direction, use %Fact~\ref{n_ex_n_uq} and induction.

%\end{proof}

\begin{corollary}
\label{uq<->ex}
$T$ simple. Assume $T$ has $k$-uniqueness for all
$3\leq k<n$  ($4\leq n$). Then the following are equivalent.
\begin{enumerate}
\item
$T$ has
$n$-uniqueness.
\item
$T$ has  $(n+1)$-amalgamation.
\item
 $B(n)$ holds in $T$.
 \end{enumerate}
 If $T$ is stable (so 2-uniqueness holds), then above equivalence  holds for $n=3$
 too.
\end{corollary}

\begin{proof}

Proposition~\ref{weak uq -> uq} says (1)$\Leftrightarrow$(3), and Corollary~\ref{n-uniqueness}  says
 (2)$\Rightarrow$(1).
%The right-to-left direction is by Corollary~\ref{n-uniqueness}.
For  (1)$\Rightarrow$(2), use Fact~\ref{n_ex_n_uq} and induction.

\end{proof}

Although 2-uniqueness is equivalent to stability, some unstable theory can have $k$-uniqueness for
$k\geq 3$ (e.g. the random graph). Thus Corollary~\ref{uq<->ex} properly covers the unstable case.

\section{Amalgamation of non-closed functors}

In this subsection, we discuss properties involving the amalgamation of functors which are not closed (i.e. $\a(s)$ could be a proper subset of $\acl(\a_s(\{i\}) : i \in s)$).  This yields both new uniqueness and new existence properties, which we will show are related to the old amalgamation properties.
In below we take $\CC_A$ for  the target category, which  is  essentially not at all different from
taking $\CC$ for that.

\begin{definition} Fix a boundedly  closed set
$A$ and let
$\CC_A=\CC_A(T)$ be the category of sets containing $A$
 with $A$-elementary embeddings.  Suppose that $S \subseteq \CP(n)$ is closed under subsets and $\A : S \rightarrow \CC_A$ is a functor.
\be
\item
For $1 \leq k < n$, a functor $\A$ is {\em $k$-skeletal (over $A$)} if

\be
\item
 $\A$ is independent; and
\item
 For every $u \in S$, $\A(u)=
\bigcup \{\bdd(\A_u(v))\mid v\subseteq u, \ |v|\leq k\}.$
\ee

\item
We say $T$ has $(n,k)$-{\em amalgamation} over $A$
if any $k$-skeletal functor $\A:\CP^-(n)\to \CC_A$  can be
extended to a $k$-skeletal functor $\hat \A:\CP(n)\to\CC_A$. Recall that
$T$ has $n$-{\em amalgamation} (over $A$) if it has
$(n,n-1)$-{\em amalgamation} (over $A$.)

\ee
\end{definition}

\begin{remark}
\label{rkonnap}
\be
\item
The notion of $(n,k)$-amalgamation $(n>k)$ above is different from that of $(n,k')$-amalgamation
$(k'>n+1)$ in \cite{Hr}.
\item
Note that for $A$-independent $\{a_0,\dots ,a_{n-1}\}$,
$$\bigcup_{\substack{v\subseteq n \\ |v|=n-1}} \bdd(a_i \mid i\in v)
\setminus
\bigcup_{\substack{w\subseteq n \\ |w|=n-2}} \bdd(a_i \mid i\in w)
$$
is the union of disjoint sets
$$
\bdd(a_i\mid i\in v)\setminus \bigcup_{\substack{w\subseteq v \\ |w|=n-2}} \bdd(a_i\mid i\in w)
$$
for $v\subseteq_{n-1} n$.
Hence it easily follows that $T$ has $n$-amalgamation
over $A$ if and only if it has $(n,n-2)$-amalgamation over $A$ (see the proof of
Theorem~\ref{relonnap} below.)

\item
Due to stationarity, any stable theory has  $(n,1)$-amalgamation.

\ee
\end{remark}

\begin{definition}
\label{relativenkuniqueness}
Let $A$ be a set.  We say that \emph{relative $(k,n)$-uniqueness holds over $A$}  ($k\leq n$) if for every $A$-independent collection of boundedly closed sets $a_0,\dots,a_{n-1}$, each containing $A$, and for any set $\{\sigma_u \mid u \subset_{k-1} n\}$ such that:
\begin{enumerate}
\item
$\sigma_u$ is an automorphism of $\overline{a}_u$, and
\item
For all $v\subsetneq u$, $\sigma_u$ is an identity on $\overline{a}_v$;
\end{enumerate}
we have that $\bigcup _{u} \sigma_{u}$ is an elementary map.

A theory $T$ has the \emph{relative $(k,n)$-uniqueness property} if it has the relative $n$-uniqueness property over any set $A$.
\end{definition}

Note that relative $(k,k)$-uniqueness is just relative $k$-uniqueness.

\begin{lemma}
\label{relknuni}
If $T$ has $B(k)$ over $A$ ($k\geq 3$), then relative $(k,n)$-uniqueness over $A$ holds  for any $n\geq k$.
\end{lemma}
\begin{proof}
By Theorem~\ref{B_n_rel_uniqueness} above, we know $B(k)$ implies  relative $(k,k)$-uniqueness.
Now for induction assume relative $(k,n)$-uniqueness over $A$.
To show  relative $(k,n+1)$-uniqueness, suppose that an $A$-independent set $\{a_0,\dots,a_{n}\}$
($a_i=\ov{a_iA}$) and maps $\{\sigma_u \mid u \subset_{k-1} n+1\}$ are given such that
\begin{enumerate}
\item
$\sigma_u$ is an automorphism of $\overline{a}_u$; and
\item
For all $v\subsetneq u$, $\sigma_u$ is an identity on $\overline{a}_v$.
\end{enumerate}
We want to show that  $\sigma=\bigcup \sigma_{u}$
is elementary.  Let $\ell$ be the least number such that $2 \ell > k$, and for $i<\ell$, let
$$
S_i:=\{u\subset_{k-1}n+1 \mid  \{2i,2i+1\}\not \subseteq u\}
$$
%$S'_i:=\{u\in S_i\mid  2i\mbox{ or } 2i+1\in u\}$,
and let $U_i:=\{u\subset_{k-1}n+1 \mid  \{2i,2i+1\} \subseteq u\}$. Also
let $B_{i}:=\bigcup\{\ov{a_u}\mid  u\in S_i\}$ and let
$C_{i}:=\bigcup\{\ov{a_u}\mid  u\in U_i\}$. Note that $\dom(\sigma)=B_i\cup C_i$.
Now due to $B(k)$ over $A$,
for any $i$ and $v$ such that $\{2i,2i+1\} \subseteq v\subset_k n+1$,
\begin{multline*}
\bigcup \{\sigma_u\mid  u\in S_i\cap\bigcap_{j<i} U_j \mbox{ and }u\subsetneq v \}\\
\cup \bigcup \{\id_{\ov{a_u}}\mid  u\in U_i\cup(S_i \setminus \bigcap_{j<i} U_j),\mbox{ and }u\subsetneq v\}
\end{multline*}
is elementary (note that $S_i\cap\bigcap_{j<i} U_j =S_i$ when $i=0$.)
 Hence by applying the induction hypothesis to the $n$-element independent set
$I_{i} := \{a_0, \ldots, \overline{a}_{2i, 2i + 1}, \ldots, a_n\}$,
 %for $$B_{ij}:=\bigcup\{\ov{a_u}|\{i,j\}\not \subseteq u\subset_{k-1}n+1\}$$
%and $$C_{ij}:=\bigcup\{\ov{a_u}|\{i,j\} \subseteq u\subset_{k-1}n+1\},$$
 we have that
$$
\mu_{i}:=\sigma\upharpoonright (B_{i}\cap\bigcap_{j<i}C_j)\cup
 \ \id_{C_{i}\cup(B_i \setminus \bigcap_{j<i}C_j)}
$$
is elementary. Also note that $\dom(\sigma)=\dom(\mu_{i})=B_{i}\cup C_{i}$.
It now follows that
$$
\mu_{i}\circ\mu_{i-1}\cdot\cdot\cdot\circ\mu_0
=\sigma\upharpoonright (\dom\sigma \setminus \bigcap\{C_j\mid j\leq i\})\cup \id_{\bigcap\{C_j \mid j\leq i\}}.$$
Hence
$\sigma=\mu_{\ell-1}\circ\mu_{\ell-2}\cdot\cdot\cdot\circ\mu_0$,
and  $\sigma$ is elementary.
\end{proof}

\begin{theorem}
\label{relonnap}
Let $T$ be a simple theory.
\be
\item If $T$ has $n$-amalgamation over $A$, then it has $(n,n-3)$-amalga\-ma\-tion
over $A$.
\item
 Assume $T$ has  $B(k+1)$ over $A$. Then
 for $n>k$,  $T$ has $(n,k)$-amalgamation over $A$ if and only if it has $(n,k-1)$-amalgamation
 over $A$.
\ee
\end{theorem}

\begin{proof}
(1) Assume $T$ has $n$-amalgamation over $A$. So it has $(n,n-2)$-amalgamation
over $A$. To show $(n,n-3)$-amalgamation, assume
an $(n-3)$-skeletal functor $\A:\CP^-(n)\to \CC_A$ is given. Due to $(n,n-2)$-amalgamation,
it suffices to show that we can extend $\A$ to an $(n-2)$-skeletal functor $\hat \A:\CP^-(n)\to \CC_A$. But this follows by a similar reason as in Remark~\ref{rkonnap} above. Namely,
for $u\subseteq n$ with $|u|=n-3$, let $\hat \A \upharpoonright \{u\mid u\subset_{\leq n-3}n\}
= \A\upharpoonright \{u\mid  u\subset_{\leq n-3}n\}$. Then for   $v\subset w\subset n$ with $|v|=n-2$ and $|w|=n-1$, simply take $\hat \A(v):=\bdd(\A(v))$ and the elementary map $\hat \A_{v,w}$ to be \emph{any} extension of $\A_{v,w}$.  To see that $\hat \A$ forms a functor, it suffices to check that
for $u\subset v\subset w\subset n$ with $|u|\leq n-3, |v|=n-2, |w|=n-1$, we have $\hat \A_{v,w} \circ \hat \A_{u,v}=\hat \A_{u,w}$. But $\hat \A_{u,v}= \A_{u,v}, \hat \A_{u,w} = \A_{u,w}$, and $\hat  \A_{v,w} \circ \A_{u,v}= \A_{u,w}$.  Hence $\hat \A$ is an $(n-2)$-skeletal functor.

(2) ($\Rightarrow$) Here is the basic idea: given some $(k-1)$-skeletal functor $\A: \P^-(n) \rightarrow \CC_A$ which we want to amalgamate, we convert it into a $k$-skeletal functor $\hat \A$ which ``extends'' $\A$ in the natural sense, and then we can amalgamate $\hat \A$, which will naturally lead to  solution for $\A$ as well.  Although it is clear what the sets $\hat \A(u)$ should be, it is less clear that we can extend the transition maps in $\A$ to $\hat \A$ in a coherent manner, but the hypothesis of $B(k+1)$ (plus  Lemma~\ref{relknuni}) ensures that we can do this.

We assume $(n,k)$-amalgamation and $B(k+1)$ over $A$, and let
a $(k-1)$-skeletal functor $\A:\CP^-(n)\to \CC_A$.
We will show that $\A$ extends
to a $k$-skeletal functor $\hat \A:\CP^-(n)\to \CC_A$.  For
$u\subset  n$, we let $$\hat \A(u)=\bigcup
\left\{\overline{\A_u(\{i_1\}), \ldots, \A_u(\{i_k\})} \mid  \{i_1, \ldots,i_k\}\subseteq u\right\}.$$

To finish the proof, we need to define the transition maps $\hat \A_{u,v}$, and for this it suffices to show:

\begin{claim}
For any $m$ such that $1 \leq m \leq n - k - 1$, there is an independent functor $\hat \A^m : \CP_{k+m}(n) \rightarrow \CC_A$ (where $\CP_{k+m}(n)$ is the set of all $u \subseteq n$ such that $|u| \leq k + m$) such that:
\begin{enumerate}
\item $\hat \A^m (u) = \hat \A(u)$ for any $u \in \CP_{k+m}(n)$; and
\item $\hat \A^m$ extends $\A$, in the sense that whenever $u \subseteq v \subseteq w \subseteq n$ and $|u| \leq k - 1$, then $$\hat \A^m_{v,w} \upharpoonright \A_v(u) = \A_{v,w} \upharpoonright \A_v(u).$$
\end{enumerate}
\end{claim}

\begin{proof}
For the base case, if $u\subset_{k} v\subset_{k+1} n$, let $\hat \A^1_{u,v}: \hat \A(u)\to \hat \A(v)$ be any elementary map extending $\A_{u,v} : \A(u)\to \A(v)$. Then by a similar argument as in the proof of (1), it follows that $\hat \A^1:\{v\mid  v\subset_{\leq k+1} n\}\to \CC_A$ is a $k$-skeletal functor.
For the induction step, assume $\hat \A^m$ extends $\A$ on
$\{v\mid  v\subset_{\leq k+m} n\}$ where $k+m<n$.
We want to extend $\hat \A^m$ to  $\{w_0,\dots ,w_{\ell}\}=\{w\mid  w\subset_{k+m+1} n\}$.
%Now let say
%$w_0=k+m+1=\{0,...,k+m\}$  and let
Fix some $s$ such that $0\leq s\leq \ell$ and let
$\{v_0,\dots ,v_r\}=\{v\mid v\subset_{k+m} w_s\}$.  For $i$ such that $0 \leq i \leq r$, we will define the maps $\hat \A^{m+1}_{v_i, w_s}$ by a recursion on $i$.
%where $v_i=w_0-\{i\}$.
First we let $\hat \A^{m+1}_{v_0,w_s}$ be an arbitrary elementary map extending $\A_{v_0,w_s}$.
By induction, assume $\hat \A^{m+1}_{v_0,w_s}, \ldots,\hat \A^{m+1}_{v_{i},w_s}$ have been compatibly chosen (for some $i < r$).  We shall define $\hat \A^{m+1}_{v_{i+1},w_s}$. Notice that for $j\leq i$,
it should be compatible with $\hat \A^{m+1}_{v_j\cap v_{i+1},w_s}$, which  is already determined as
$\hat \A^{m+1}_{v_j,w_s}\circ \hat \A^m_{v_j\cap v_{i+1},v_j}.$
Hence first consider an arbitrary  elementary map
$f:\hat \A (v_{i+1})\to\hat \A(w_s)$ extending $\A_{v_{i+1},w_s}$.
Now for each $j \leq i$, there is elementary $\mu_j:\hat \A^m_{w_s}(v_j\cap v_{i+1})
\to \hat \A^m_{w_s}(v_j\cap v_{i+1})$ such that
$$\mu_j\circ f\circ \hat \A^m_{v_j\cap v_{i+1}, v_{i+1}} \upharpoonright \hat \A (v_j\cap v_{i+1})
=  \hat \A^m_{v_j\cap v_{i+1}, w_{s}} \upharpoonright \hat \A(v_j\cap v_{i+1}).$$ Due to Lemma~\ref{relknuni},
$\mu:=\bigcup \{\mu_j\mid  j\leq i\}\cup  \id_B$
is elementary where $B=\hat \A(w_s)\setminus \bigcup_j \dom(\mu_j)$. Then we let
$\hat \A^{m+1}_{v_{i+1},w_s} = \mu\circ f$, which extends $\A_{v_{i+1},w_s}$. We have defined all $\hat \A_{v_i,w_s}$ for $i\leq k+m$.
Then for $v\subset v_i$, let $\hat \A^{m+1}_{v,w_s} = \hat \A^{m+1}_{v_i,w_s}\circ \hat \A^m_{v,v_i}$. By the  construction, this does not depend on the choice of $v_i$.  Then
for $u\subset v(\subseteq v_i)\subset w_s$,
$$\hat \A^{m+1}_{u,w_s}=\hat \A^{m+1}_{v_i,w_s}\circ \hat \A^m_{u,v_i}$$ $$=
\hat \A^{m+1}_{v_i,w_s}\circ \hat \A^m_{v,v_i}\circ \hat \A^m_{u,v}=\hat \A^{m+1}_{v,w_s}\circ\hat \A^{m+1}_{u,v}.$$
Therefore $\hat \A^{m+1}:\CP(w_s)\to \CC_A$ forms a $k$-skeletal functor.
Now  the union of the functors $\hat \A^{m+1}:\bigcup_{s\leq \ell}\CP(w_s)=\{w\mid w\subset_{\leq k+m+1}n\}\to \CC_A$ clearly forms a $k$-skeletal functor extending $\hat \A^m$, as desired.

\end{proof}

($\Leftarrow$) For this direction, the idea is as follows: we assume that $(n, k-1)$-amalgamation and $B(k+1)$ both hold over $A$ and that we are given a $k$-skeletal functor $\hat \A: \CP^-(n) \rightarrow \CC_A$ which we want to amalgamate.  The functor $\hat \A$ has a canonical ``restriction'' to a $(k-1)$-skeletal functor $\A: \CP^-(n) \to \CC_A$, and by $(n, k-1)$-amalgamation, this has a solution $\B: \CP(n) \to \CC_A$.  Then we extend $\B$ to a $k$-skeletal $\hat \B$, and just as in the argument for $\Rightarrow$ above, we can use $B(k+1)$ to ``untwist'' $\hat \B$ as necessary to ensure that it is a solution to the original amalgamation problem $\hat \A$.

We now give some details.  Let $\hat \A$, $\A$, and $\B$ be as in the previous paragraph.  It is clear what set $\hat \B(n)$ needs to be, and for $u \subseteq v \subsetneq n$, we need to have $\hat \B_{u,v} = \A_{u,v}$, so we only need to define $\hat \B_{u,n}$ for each $u \subsetneq n$. We shall first define $\hat \B(w_i,n)$, where $\{w_0, \ldots, w_{n-1}\}=\{w\mid w\subset_{n-1}n\}$, such that $\hat \B(w_i, n)$ extends $\B(w_i,n)$.  We proceed by induction on $i < n$.  For the base case, we may let $\hat \B(w_0,n)$
be any appropriate extension of $\B(w_0,n)$.  By induction, we assume that $\hat \B(w_0,n),\dots ,\hat \B(w_i,n)$ are compatibly chosen and $i+1 < n$. Then for $j \leq i$ and  $w_j\cap w_i(\subset_{n-2}n)$, $\hat \B(w_{i+1},n)$ should be compatible with $\hat \B(w_j\cap w_{i+1},n)$, which must equal $\hat \B(w_j,n)\circ \hat \A(w_j\cap w_{i+1},w_j).$
For this, consider an arbitrary  elementary map
$f:\hat \A(w_{i+1})\to \hat \B(n)$ extending $ \B(w_{i+1},n)$. As in the proof of $\Rightarrow$,
for each $j \leq i$, there is an elementary $\mu_j:\hat \B_{n}(w_j\cap w_{i+1})
\to \hat \B_{n}(w_j\cap w_{i+1})$ such that
$$\mu_j\circ f\circ \hat \A(w_j\cap w_{i+1}, w_{i+1})\upharpoonright \hat \A(w_j\cap w_{i+1})
=  \hat \B(w_j\cap w_{i+1}, n)
\upharpoonright \hat \A(w_j\cap w_{i+1}).$$
Then due to Lemma~\ref{relknuni},
$\mu:=\bigcup \{\mu_j\mid j\leq i\}\cup  \id_B$
is an elementary map, where $B=\hat \B(n)\setminus \bigcup_j \dom(\mu_j)$.
Then we let
$\hat \B(w_{i+1},n)=\mu\circ f$,
which extends $\B(w_{i+1},n)$.

To finish, we need to define $\hat \B(u,n)$ for $u$ an arbitrary subset of $n$.  If $u \subseteq w_i$, then we simply let $\hat \B(u,n) = \hat \B(w_i, n) \circ \hat \A(u, w_i)$.  That this is well-defined, and that this yields a functor, follow exactly as in the proof of $\Rightarrow$.
\end{proof}

\begin{corollary}
\label{Bn_simplicity}
\be
\item
  If $T$ is $B(n-2)$-simple having $n$-amalgamation, then
 it has $(n,k)$-amalgamation for each $0<k<n$. Conversely if
 $T$ is $B(n-1)$-simple having $(n,1)$-amalgamation, then it has
 $(n,k)$-amalgamation for each $0<k<n$.
 \item If $T$ is $B(k)$-stable (i.e. stable having $B(3),\dots,B(k)$), then $(n,k-1)$-amalgamation
 holds for $n\geq k$.
 \ee
\end{corollary}

\section{Examples}

This section points out the obstacles in generalizing the construction of a definable groupoid from failure of $3$-uniqueness to simple unstable theories.

\subsection{\boldmath $B(3)$ can fail while 4-existence holds}

Here we give an example of a simple (in fact, supersimple $SU$-rank 1) theory with a symmetric witness to failure of $3$-uniqueness for which $4$-existence still holds. The example shows that (1) failure of $3$-uniqueness (or even failure of $B(3)$) need not imply the failure of 4-existence if 2-uniqueness does not hold; and (2) without stability, failure of $B(3)$ on a Morley sequence does not imply that there is a definable groupoid whose objects are the elements of the sequence.

Each model contains three sorts:
an infinite set $I$, the set of 2-element subsets $K:=[I]^2$,
and a double cover $G^*$ of $K$. We have the membership function,
so that the sets $I$ and $K$ are interdefinable, and the elements of
$G^*$ have the form $\langle \{a,b\},\delta\rangle$,
where $a,b\in I$ and $\delta=0$ or $\delta=1$. There is a projection
$\pi$ from $G^*$ to $K$. It is easy to see that this part of the
example is totally categorical and has the natural quasi-finite axiomatization $T_1$.

Finally, we have a ternary predicate $\Q$ satisfying the following axioms:
\begin{enumerate}
\item
$\Q(x,y,z)$  implies that $x,y,z\in G^*$ and that  $\{\pi(x),\pi(y),\pi(z)\}$ are \emph{compatible}, i.e., form all the 2-element subsets of a 3-element set;
\item
$\Q$ is symmetric with respect to all the permutations of its arguments;
\item
if $\pi(x)$, $\pi(y)$, and $\pi(z)$ form a compatible triple and if
$x'$ is the other element in the fiber above $\pi(x)$, then
$$
\Q( x,y,z)\iff \lnot \Q( x',y,z).
$$
\item
for each $2\le n<\omega$ and every $W\subset [n]^2$ the following holds:
\begin{multline*}
\forall a_0,\dots a_{n-1}, \in I\ \forall x_{ij} \in \pi^{-1}(a_i,a_j)
\\
 \exists b_W\in I\setminus \{a_0,\dots,a_{n-1}\}\ \exists y_i^W \in \pi^{-1}(a_i,b_W)
\\
\bigwedge_{(i,j)\in W} \Q(x_{ij},y^W_i,y^W_j)
\land
\bigwedge_{(i,j)\notin W} \lnot \Q(x_{ij},y^W_i,y^W_j).
\end{multline*}
\end{enumerate}

Let $T$ be the theory $T_1$ together with the axioms (1)--(4) above.

\begin{claim}\label{cl1}
The theory $T$ is consistent.
\end{claim}

\begin{proof}
We build a structure that satisfies $T$. We begin with an (infinite) structure $M_1$ that satisfies $T_1$. Recall that, for $a\ne b\in I$ an element $x$ in the fiber $\pi^{-1}(ab)$ has the form $\langle a,b,\delta\rangle$, where $\delta\in \{0,1\}$.

Let $R$ be a random ternary graph on $I(M_1)$. Now we define $Q$. For each three-element subset $\{a_0,a_1,a_2\}$ of $I$, for elements $x_{ij}\in \pi^{-1}(a_ia_j)$, $0\le i<j\le 2$, define $Q(x_{01},x_{02},x_{12})$ to hold if and only if one of the following holds:
\begin{gather*}
\lnot R(a_0,a_1,a_2) \textrm{ and } \delta_{01}+\delta_{02}+\delta_{12}=0
\textrm{ modulo 2}
\\
\textrm{or}
\\
R(a_0,a_1,a_2) \textrm{ and } \delta_{01}+\delta_{02}+\delta_{12}=1
\textrm{ modulo 2}.
\end{gather*}
Let $M$ be the reduct of the above structure to the language of $T$ (that is, we remove the predicate $R$ from the language). It is obvious that the axioms (1)--(3) hold in $M$. Checking that (4) holds is also routine.
\end{proof}

\begin{notation}
Let $M$ be a model of $T$ and let $A\subset I(M)$. We call the set $\pi^{-1}([A]^2)$ a \emph{closed substructure of $M$ generated by $A$}.
\end{notation}

The following property is immediate from the axioms of $T$:

\begin{claim}\label{cl2}
Let $M$ be a model of $T$. Then for any $A\subset I(M)$, for any choice of elements $x_{ab}\in \pi^{-1}(a,b)$, $a,b\in A$, the isomorphism type of the closed substructure generated by $A$ is uniquely determined by the isomorphism type of the set $\{x_{ab}\mid a,b\in A\}$.

That is, if $A$ is a subset of $I(M)$, $X_A =\{x_{ab}\in \pi^{-1}(a,b)\mid a\ne b\in A\}$, and $f$ is a partial elementary map defined on $A\cup X_A$, then $f$ can be uniquely extended to a partial elementary map on the closed substructure generated by $A$.
\end{claim}

Now an easy back-and-forth argument gives

\begin{proposition}
The theory $T$ is $\omega$-categorical.
\end{proposition}

\begin{proof}
Let $M$ and $N$ be countable models of $T$. Enumerate the sets $I(M)$ and $I(N)$ as $\{a_i\mid i<\omega\}$ and $\{b_i\mid i<\omega\}$ respectively. Define a chain of partial isomorphisms $f_i:M\to N$ such that
\begin{enumerate}
\item
$f_0(a_0)=b_0$;
\item
for each $i<\omega$, $f_i$ is an isomorphism between finite closed substructures of $M$ and $N$;
\item
for each $i<\omega$, the domain of $f_{2i}$ contains the set $\{a_0,\dots,a_i\}$ and the range of $f_{2i-1}$ contains the set $\{b_0,\dots,b_i\}$.
\end{enumerate}

It is straightforward to define the partial isomorphisms using Claim~\ref{cl2} and Axiom (4) above.
\end{proof}

Formally, the projection function $\pi$ is defined only on the elements of $G^*$. In what follows, it is convenient to treat $\pi$ as a total function on the universe of $\C$, defining $\pi(a)=a$ for all $a\in I(\C)$.

\begin{claim}\label{cl3}
Let $\C$ be the monster model of $T$. Then

(1) for any $C\subset \C$, the algebraic closure of $C$ is the closed substructure of $\C$ generated by $\pi(C)$;

(2) the algebraic closure of a finite set is finite;

(3) for any three algebraically closed sets $A$, $B$, and $C$ in $\C$, we have $\acl(ABC)=\acl(AB)\cup \acl(AC)\cup \acl(BC)$.
\end{claim}

\begin{proof}
(1) Since the union of a directed system of closed substructures is a closed substructure and $\acl(C)=\bigcup _{c\subset_{fin} C}\acl(c)$, it is enough to prove the statement for finite sets $C$.

Let $\bar C$ be the closed substructure generated by a finite set $C$. It is clear that $\bar C\subset \acl(C)$. To show the reverse inclusion, we note that any element of $\C$ not contained in $\bar C$ is not algebraic over $C$.

The statement (2) is immediate from (1).

(3) It is easy to check that the union of closed structures $\acl(AB)\cup \acl(AC)\cup \acl(BC)$ is itself closed.
\end{proof}

\begin{proposition}\label{simple}
The theory $T$ is supersimple of $SU$-rank 1.
\end{proposition}

\begin{proof}
We define the notion of independence in a monster model of $T$, and show that this notion satisfies the axioms of independence (Definition~4.1 of~\cite{KP}) as well as the Independence Theorem. By Theorem~4.2 of~\cite{KP}, this will show that $T$ is simple and that the notion of independence must coincide with non-forking. The latter part will allow to show that the $SU$-rank is equal to 1.

%Let $A$ be an arbitrary subset of the monster model $\C$, recall that $\acl(A)$ is a closed substructure of $\C$.

For $A,B,C\subset \C$, define $A\nonfork_C B$ if and only if $\acl(AC)\cap \acl(BC)\subset \acl(C)$.

It is immediate that the relation $\nonfork_{}$ satisfies the invariance, local character, finite character, symmetry and transitivity. Note that the local character is finite: for any finite tuple $\bar a$ and a set $B$, the set $\acl(\bar a)\cap \acl(B)$ is finite by Claim\ref{cl3}.

The extension property and Independence Theorem follow from compactness and Axiom (4). The above shows that $T$ is supersimple. It remains to note that every 1-type that forks is algebraic over its domain. So $T$ has $SU$-rank 1.
\end{proof}

\begin{proposition}
The theory $T$ has weak elimination of imaginaries.
\end{proposition}

\begin{proof}
By Proposition~2.2 of~\cite{Ik07}, it suffices to show the following: for all finite tuples $a$ and sets $A$, $B$, if $a\nonfork_A B$ and $a\nonfork _B A$, then
$$
a\nonfork _{\acl(A)\cap \acl(B)} AB.
$$

To make the notation less cumbersome, we use the symbols $\bar A$ and $\bar B$ for $\acl(A)$ and $\acl(B)$, respectively. Using the description of non-forking from Proposition~\ref{simple}, we need to show that if
$$
\acl(aA)\cap \acl(AB)\subset \bar A \textrm{ and }
\acl(aB)\cap \acl(AB)\subset \bar B,
$$
then $\acl(a (\bar A\cap \bar B))\cap \acl(AB)\subset \bar A\cap \bar B$.

\begin{claim}\label{cl4}
For all sets $A$ and $B$, for any finite tuple $a$, we have
$\acl(a(\bar A \cap \bar B)=\acl(aA)\cap \acl(aB)$.
\end{claim}

\begin{proof}
Note that $\acl(X\cap Y)$ can be a proper subset of $\acl(X)\cap \acl(Y)$ for $X,Y\subset \C$ if $X$ and $Y$ are not algebraically closed, so some care is needed to carry out the argument.

It is enough to check that $\pi(a(\bar A \cap \bar B))=\pi(aA)\cap \pi(aB)$. We have
\begin{multline*}
\pi(a(\bar A \cap \bar B))=\pi(a)\cup (\pi(\bar A) \cap \pi(\bar B))
\\
=(\pi(a)\cup \pi(A)) \cap (\pi(a)\cup \pi(B))=\pi(aA)\cap \pi(aB).
\end{multline*}
The first and last equalities hold simply because $\pi$ is a function; and the middle equality follows from $\pi(A)=\pi(\bar A)$.
\end{proof}

With Claim~\ref{cl4}, it suffices to prove
$\acl(aA) \cap \acl(aB) \cap \acl(AB) \subset \bar A\cap \bar B$. This obviously follows now from the assumptions.
\end{proof}

\begin{proposition}
The theory $T$ has 4-amalgamation.
\end{proposition}

\begin{proof}

Given three algebraically closed subsets $A_0$, $A_1$, and $A_2$, consider three algebraically closed structures $acl(A_iA_jB_{ij})$, $0\le i<j\le 2$, that are compatible ``along the edges''. That is, $B_{ij}\equiv_{A_i} B_{ik}$ for all $i$, $j$, and $k$ (in particular, $B_{ij}$ realize the same type over the empty set and we assume that $B_{ij}$ are algebraically closed). We need to show that there is a structure $B$ such that $\acl(A_iA_jB)\equiv _{A_iA_j} \acl(A_iA_jB_{ij})$ for all $i$, $j$.

By compactness, it is enough to consider finite sets $A_i$ and $B_{ij}$.
Moreover, it is enough to establish amalgamation for the case when $B_{ij}=\{b_{ij}\}$, for some $b_{ij}\in I$. The latter is immediate from axiom (4) of $T$.
\end{proof}

\begin{remark}
It is obvious that $T$ fails the property $B(3)$. Indeed, for any three points $a,b,c\in I$, the set $\acl(ac)$ is definable from $\acl(ab)\cup \acl(bc)$, yet the points in the fiber $\pi^{-1}(ab)$ are algebraic over $ab$.

Thus, the example shows that, in a simple unstable theory, the failure of property $B(3)$ need not imply the failure of 4-amalgamation.
\end{remark}

Note that in the theory $T$ described above, the sort $I$ is \emph{not} an indiscernible set.  In fact, a $4$-element subset of $I$ can have two distinct types.

%\begin{proposition}
%The theory $T$ has (???)-uniqueness
%\end{proposition}

\begin{remark}
The example described above can be generalized to higher dimensions.  That is, for every $n \geq 3$, there is an analogous simple theory $T_n$ such that $B(n)$ fails, but $(n+1)$-existence holds.
\end{remark}

\subsection{\boldmath $B(3)$ may fail over models}

The example shows that, in a simple unstable theory, the property $B(3)$ does not have to hold over models. Most of the calculations are very similar to (and easier than) those for the first example, so we just state the results.

The basic structure is the same as in the previous example.
Each model contains a sort $I$, the sort $K$, where $K$ is the set of 2-element subsets of $I$, and a double cover $G^*$ of $K$. We have the membership function, and think of the elements of $G^*$ as of triples $\langle \{a,b\},\delta\rangle$, where $a,b\in I$ and $\delta=0$ or $\delta=1$. There is a projection $\pi$ from $G^*$ to $K$.

We have two ternary predicates: $R$, defined on 3-element subsets of $I$ and $Q$, such that:
\begin{enumerate}
\item
$R$ is a generic predicate on $I$;
\item
$\Q(x,y,z)$  implies that $x,y,z\in G^*$, that  $\{\pi(x),\pi(y),\pi(z)\}$ are \emph{compatible}, and that $R$ holds on the corresponding 3-element set $\pi(x)\cup \pi(y)\cup \pi(z)\}$;
\item
$\Q$ is symmetric with respect to all the permutations of its arguments;
\item
if $\pi(x)$, $\pi(y)$, and $\pi(z)$ form a compatible triple, $R$ holds on $\pi(x)\cup \pi(y)\cup \pi(z)\}$, and if
$x'$ is the other element in the fiber above $\pi(x)$, then
$$
\Q( x,y,z)\iff \lnot \Q( x',y,z).
$$
\end{enumerate}

Let $T$ be the theory describing the basic structure together with the axioms (1)--(4) above.

\begin{claim}\label{cl1-2}
The theory $T$ is consistent.
\end{claim}

\begin{proof}
We build a structure that satisfies $T$. We begin with an (infinite) structure $M$ that satisfies the basic axioms and define a random ternary graph on $I(M)$.

For each three-element subset $\{a_0,a_1,a_2\}$ of $I$, if $M\models R(a_0,a_1,a_2)$, then, for elements $x_{ij}\in \pi^{-1}(a_ia_j)$, $0\le i<j\le 2$, define $Q(x_{01},x_{02},x_{12})$ to hold if and only if $\delta_{01}+\delta_{02}+\delta_{12}=0$ modulo 2. Define $Q$ to fail on the elements in other fibers. It is clear that $M$ is a needed model.
\end{proof}

A standard back-and-forth construction gives the following.

\begin{claim}\label{cl2-2}
The theory $T$ is $\omega$-categorical.
\end{claim}

\begin{claim}\label{cl3-2}
The theory $T$ is supersimple of $SU$-rank 1.
\end{claim}

The argument for Claim~\ref{cl3-2} is similar to that of Proposition~\ref{simple} (the independence relation is exactly the same).

\begin{claim}\label{cl4-2}
The theory $T$ fails $B(3)$, even over models and it fails the 4-existence property over models.
\end{claim}

\begin{proof}
Let $M\models T$ and choose $a_1,a_2,a_3\in I$ such that $\models R(a_1,a_2,a_3)$, but for no element $a\in I(M)$ we have $\models R(a,a_i,a_j)$, $1\le i<j\le 3$. It is clear that the elements $a_1$, $a_2$, and $a_3$ witness the failure of $B(3)$. Failure of 4-existence is established similarly.
\end{proof}

\end{document}